\documentclass[11pt,a4paper,oneside,reqno]{amsart}

\usepackage{latexsym}
\usepackage{amsmath}
\usepackage{amsthm}
\usepackage{amssymb}
\usepackage{amsfonts}
\usepackage{fancyhdr}
\usepackage{comment}

\usepackage{mathrsfs}                           

\usepackage{calc}
             {\begin{list}{\arabic{enumi}.}{\usecounter{enumi}%
              \setlength{\labelsep}{0.5em}%
              \settowidth{\labelwidth}{\arabic{enumi}.}%
              \setlength{\leftmargin}{\labelwidth+\labelsep}}}%
             {\end{list}}

\def \ba {\begin {eqnarray*} }
\def \ea {\end {eqnarray*} }
\def \beq {\begin {eqnarray}}
\def \eeq {\end {eqnarray}}

\def \p {\partial}

\newcommand{\tr}{\mathrm{tr}}

\newcommand{\mR}{\mathbb{R}}                    
\newcommand{\mC}{\mathbb{C}}                    
\newcommand{\R}{\mathbb{R}}                    
\newcommand{\mN}{\mathbb{N}}                    
\newcommand{\abs}[1]{\lvert #1 \rvert}          
\newcommand{\norm}[1]{\lVert #1 \rVert}         
\newcommand{\br}[1]{\langle #1 \rangle}         

\newcommand{\eps}{\varepsilon}

\newcommand{\mF}{\mathscr{F}}

\newcommand{\id}{\mathrm{Id}}

\newcommand{\re}{\mathrm{Re}}
\newcommand{\im}{\mathrm{Im}}
\newcommand{\supp}{\mathrm{supp}}

\newcommand{\op}[1]{\mathop{\rm Op}\nolimits_{#1}}
\newcommand{\WF}{\mathop{\rm WF}\nolimits_{\rm scl}}

\theoremstyle{plain}
\newtheorem{thm}{Theorem}[section]
\newtheorem{prop}[thm]{Proposition}

\newtheorem{lemma}[thm]{Lemma}

\newtheorem{question}{Question}[section]
\newtheorem*{conjecture}{Conjecture}

\theoremstyle{definition}
\newtheorem*{definition}{Definition}
\newtheorem{example}{Example}[section]
\newtheorem*{remark}{Remark}

\numberwithin{equation}{section}

\title[Calder\'on problem in transversally anisotropic geometries]{The Calder\'on problem in transversally anisotropic geometries}

\author{David Dos Santos Ferreira}
\address{Institut \'Elie Cartan, UMR 7502, Universit\'e de Lorraine, CNRS, Campus des Aiguillettes B.P. 70239, F-54506 Vandoeuvre-l\`es-Nancy, France}
\email{ddsf@math.cnrs.fr}
\author{Yaroslav Kurylev}
\address{Department of Mathematics, University College London}
\email{y.kurylev@ucl.ac.uk}
\author{Matti Lassas}
\address{Department of Mathematics and Statistics, University of Helsinki}
\email{matti.lassas@helsinki.fi}
\author{Mikko Salo}
\address{Department of Mathematics and Statistics, University of Jyv\"askyl\"a}
\email{mikko.j.salo@jyu.fi}

\begin{document}

\begin{abstract}
We consider the anisotropic Calder\'on problem of recovering a conductivity matrix or a Riemannian metric from electrical boundary measurements in three and higher dimensions. In the earlier work \cite{DKSaU}, it was shown that a metric in a fixed conformal class is uniquely determined by boundary measurements under two conditions: (1) the metric is conformally transversally anisotropic (CTA), and (2) the transversal manifold is simple. In this paper we will consider geometries satisfying (1) but not (2). The first main result states that the boundary measurements uniquely determine a mixed Fourier transform / attenuated geodesic ray transform (or integral against a more general semiclassical limit measure) of an unknown coefficient. In particular, one obtains uniqueness results whenever the geodesic ray transform on the transversal manifold is injective. The second result shows that the boundary measurements in an infinite cylinder uniquely determine the transversal metric. The first result is proved by using complex geometrical optics solutions involving Gaussian beam quasimodes, and the second result follows from a connection between the Calder\'on problem and Gel'fand's inverse problem for the wave equation and the boundary control method.
\end{abstract}

\maketitle
\setcounter{tocdepth}{1} 
\tableofcontents

\section{Introduction} \label{sec_intro}

The anisotropic Calder\'on problem consists in determining the electrical conductivity matrix of a medium, up to a change of coordinates, from current and voltage measurements made at the boundary. More generally the problem may be posed on a smooth Riemannian manifold with boundary. In this case the question is to determine the geometric structure of the manifold from the Cauchy data of harmonic functions. The purpose of this paper is to study the anisotropic Calder\'on problem in transversally anisotropic geometries, where the manifold admits a distinguished Euclidean direction, and to prove uniqueness results for inverse problems in this setting.

Let $(M,g)$ be a compact oriented Riemannian manifold with smooth boundary $\partial M$. Harmonic functions in $M$ are solutions of the Laplace-Beltrami equation 
$$
\Delta_g u = 0 \quad \text{in } M.
$$
Here, the Laplace-Beltrami operator is given in local coordinates by 
$$
\Delta_g u = \abs{g}^{-1/2} \frac{\partial}{\partial x_j} \left( \abs{g}^{1/2} g^{jk} \frac{\partial u}{\partial x_k} \right)
$$
where $(g_{jk})$ is the metric in local coordinates, $(g^{jk}) = (g_{jk})^{-1}$, and $\abs{g} = \det(g_{jk})$. Here and below we are using the Einstein summation convention.

The boundary data of harmonic functions on $M$ is given by the Cauchy data set 
$$
C_g = \{(u|_{\partial M}, \partial_{\nu} u|_{\partial M}) \,;\, \Delta_g u = 0 \text{ in } M, \ u \in H^1(M) \}.
$$
The normal derivative $\partial_{\nu} u|_{\partial M} = \langle du, \nu \rangle|_{\partial M}$, where $\nu$ is the $1$-form corresponding to the unit outer normal of $\partial M$, is interpreted in the weak sense as an element of $H^{-1/2}(\partial M)$. It is clear that if $\psi: M \to M$ is a diffeomorphism satisfying $\psi|_{\partial M} = \id$, then $C_{\psi^* g} = C_g$. On manifolds of dimension $\geq 3$, the anisotropic Calder\'on problem \cite{LeU} amounts to proving that $C_g$ uniquely determines $g$ up to isometry.

\begin{conjecture}
Let $(M,g_1)$ and $(M,g_2)$ be two compact Riemannian manifolds with smooth boundary, and let $\dim(M) \geq 3$. If $C_{g_1} = C_{g_2}$, then 
\begin{align*}
g_2 = \psi^* g_1
\end{align*}
where $\psi: M \to M$ is a diffeomorphism with $\psi|_{\partial M} = \id$.
\end{conjecture}

This statement has only been proved for real-analytic metrics \cite{LeU} with topological assumptions relaxed in \cite{LTU}, \cite{LaU}, and for Einstein metrics (which are real-analytic in the interior) \cite{GS}. The general case remains a major open problem, and we refer to \cite{DKSaU} for a discussion and further references. The corresponding two-dimensional result, involving an additional obstruction arising from the conformal invariance of the Laplace-Beltrami operator, is known \cite{LaU}. See \cite{Belishev_survey}, \cite{Belishev_rings} for another interesting approach to this problem.

The work \cite{DKSaU} introduced methods for studying the anisotropic Calder\'on problem in manifolds which are not real-analytic, but where the metric has certain form. This was based on the concept of \emph{limiting Carleman weights}, introduced earlier in the Euclidean case in \cite{KSU}. One of the main results of \cite{DKSaU} states that on a simply connected open manifold, the existence of a limiting Carleman weight is equivalent to the existence of a nontrivial parallel vector field for some conformal metric. Locally, this condition is equivalent with the manifold being conformal to a product of a Euclidean interval and some $(n-1)$-dimensional manifold. We formalize this notion in two definitions:

\begin{definition}
Let $(M,g)$ be a compact oriented manifold with $C^{\infty}$ boundary. In this paper we always assume that $n = \dim(M) \geq 3$.

\begin{enumerate}
\item[(a)]
$(M,g)$ is called \emph{transversally anisotropic} if $(M,g) \subset \subset (T,g)$ where $T = \mR \times M_0$, $g = e \oplus g_0$, $(\mR,e)$ is the Euclidean line, and $(M_0,g_0)$ is some compact $(n-1)$-dimensional manifold with boundary. Here $(M_0,g_0)$ is called the \emph{transversal manifold}.
\item[(b)]
$(M,g)$ is called \emph{conformally transversally anisotropic} (CTA) if $(M,c g)$ is transversally anisotropic for some smooth positive function $c$.
\end{enumerate}
\end{definition}

Examples of CTA manifolds include compact subdomains of the model spaces $\mR^n$, sphere $S^n$ minus a point, or hyperbolic space $H^n$, compact subdomains of locally conformally flat manifolds such as 3D symmetric spaces as long as they are contained in a conformally flat coordinate neighborhood, and conformally warped products 
$$
(M,g) \subset \subset (\mR \times M_0, g), \quad g = c(e \oplus f g_0)
$$
where $f$ is a positive function depending only on the Euclidean variable in $\mR \times M_0$. If $(x_1,x')$ are local coordinates in $\mR \times M_0$, the last condition reads in terms of matrices 
$$
g(x_1,x') = c(x_1,x') \left( \begin{array}{cc} 1 & 0 \\ 0 & f(x_1) g_0(x') \end{array} \right).
$$
See \cite{DKSaU}, \cite{DKS}, \cite{LiS} for more details.

The first main theorem in this paper considers the anisotropic Calder\'on problem in a fixed conformal class. Since any conformal diffeomorphism fixing the boundary must be the identity map, there is no obstruction to uniqueness arising from isometries in this case (see \cite{Li}). The article \cite{DKSaU} gave a uniqueness result for this problem on CTA manifolds if additionally the transversal manifold $(M_0,g_0)$ is \emph{simple}, meaning that any two points in $M_0$ are connected by a unique geodesic depending smoothly on the endpoints and that $\partial M_0$ is strictly convex (its second fundamental form is positive definite). Moreover, a reconstruction procedure was given in \cite{KSaU_reconstruction} and stability estimates (of double logarithmic type) were given in \cite{CaroSalo}. The proof used the fact that the geodesic ray transform is injective on simple manifolds. On general transversal manifolds we use the following definition.

\begin{definition}
We say that the (geodesic) \emph{ray transform} on the transversal manifold $(M_0,g_0)$ is injective if any function $f \in C(M_0)$ which integrates to zero over all nontangential geodesics in $M_0$ must satisfy $f = 0$. Here, a unit speed geodesic segment $\gamma: [0,L] \to M_0$ is called \emph{nontangential} if $\dot{\gamma}(0)$, $\dot{\gamma}(L)$ are nontangential vectors on $\partial M_0$ and $\gamma(t) \in M_0^{\text{int}}$ for $0 < t < L$.
\end{definition}

\begin{thm} \label{claim1}
Let $(M,g_1)$ and $(M,g_2)$ be two CTA manifolds in the same conformal class. Assume in addition that the ray transform in the transversal manifold is injective. If $C_{g_1} = C_{g_2}$, then $g_1 = g_2$.
\end{thm}

In fact this result is a consequence of a corresponding result for the Schr\"odinger equation. Let $q \in L^{\infty}(M)$, and define the Cauchy data set for the Schr\"odinger operator $-\Delta_g + q$ by 
$$
C_{g,q} = \{ (u|_{\partial M}, \partial_{\nu} u|_{\partial M}) \,;\, (-\Delta_g + q)u = 0 \text{ in } M, \ u \in H^1(M) \}.
$$
Again, the normal derivative $\partial_{\nu} u|_{\partial M}$ is interpreted in the weak sense as an element of $H^{-1/2}(\partial M)$.

\begin{thm} \label{claim2}
Let $(M,g)$ be a CTA manifold, and let $q_1, q_2 \in C(M)$. Assume in addition that the ray transform in the transversal manifold is injective. If $C_{g,q_1} = C_{g,q_2}$, then $q_1 = q_2$.
\end{thm}

Starting from the pioneering works \cite{C}, \cite{F}, \cite{HN}, \cite{Nachman}, \cite{No_1988}, \cite{SU}, see also \cite{Novikov_survey}, \cite{U_IP}, the standard approach for proving uniqueness and reconstruction results for the Calder\'on problem is based on special \emph{complex geometrical optics} solutions to elliptic equations. The paper \cite{DKSaU} presented a construction of such solutions on CTA manifolds and proved Theorems \ref{claim1} and \ref{claim2} under the additional restriction that the transversal manifold $(M_0,g_0)$ is simple, for instance the Carleman estimates required for the construction of correction terms, were valid without this additional restriction. However, in the end the simplicity assumption was used to produce solutions that concentrate near geodesics in $(M_0,g_0)$ and also to show that the potentials can be determined by inverting the geodesic ray transform (actually with attenuation) in the transversal manifold.

In this paper we remove the simplicity assumption on the transversal manifold in the construction of complex geometrical optics solutions, and prove Theorems \ref{claim1} and \ref{claim2} on any CTA manifold for which the ray transform is injective. In cases where the ray transform is not injective, we obtain partial results (see Theorems \ref{claim3} and \ref{thm_calderon_cylinder_reconstruction}) but the problem remains open in general. Injectivity of the ray transform is known to hold in the following classes of manifolds $(M_0,g_0)$:

\begin{enumerate}
\item[(a)]
Simple manifolds of any dimension (see \cite{Sh}).
\item[(b)]
Manifolds of dimension $\geq 3$ that have strictly convex boundary and are globally foliated by strictly convex hypersurfaces (\cite{UhlmannVasy}).
\item[(c)]
A class of non-simple manifolds of any dimension such that there are sufficiently many geodesics without conjugate points and the metric is close to a real-analytic one (see \cite{StU_geodesic_nonsimple} for the precise description of this class).
\item[(d)]
Any manifold having a dense subset that is covered by totally geodesic submanifolds in which the ray transform is injective (injectivity of the ray transform follows immediately from the injectivity in the totally geodesic submanifolds). Examples include subdomains of $(N_1 \times N_2, h_1 \oplus h_2)$ where $(N_1,h_1)$ has injective ray transform and $(N_2,h_2)$ is any manifold.
\item[(e)]
There are counterexamples to injectivity of the ray transform. The standard one is the sphere with a small cap removed: any function on the sphere that is odd with respect to the antipodal map and vanishes near the removed cap integrates to zero over nontangential geodesics. See also \cite{BaoZhang}, \cite{MSU}, \cite{StU_geodesic_fold} for microlocal analysis of the ray transform in non-simple geometries.
\end{enumerate}

In fact, Theorems \ref{claim1} and \ref{claim2} involving the ray transform will be obtained as a special case from a more general complex geometrical optics construction on CTA manifolds. If $(M,g)$ is a CTA manifold, so $(M,g) \subset \subset (\mR \times M_0, g)$ for some compact manifold $(M_0,g_0)$ where $g = c(e \oplus g_0)$, we denote points on $M$ by $x = (x_1,x')$ where $x_1$ is the Euclidean variable and $x' \in M_0$. If $q \in L^{\infty}(M)$, we will consider solutions of the Schr\"odinger equation $(-\Delta_g+q)u = 0$ in $M$ of the form 
$$
u(x) = e^{s x_1} c(x)^{-\frac{n-2}{4}}(v_s(x) + r_s(x))
$$
where $s$ is a slightly complex large frequency, 
$$
s = \tau + i\lambda,
$$
where the real parameter $\tau$ will tend to infinity while $\lambda \in \mC$ is fixed, and where $v_s = v_s(x') \in C^2(M_0)$ is a \emph{quasimode}, or \emph{approximate eigenfunction}, with frequency $s$ in the transversal manifold. The correction term $r_s$ will satisfy $\norm{r_s}_{L^2(M)} \to 0$ as $\tau \to \infty$. The concentration properties of the quasimodes $v_s$ in the high frequency limit as $\tau \to \infty$ will be crucial in determining properties of the potential.

\begin{definition}
If $\lambda \in \mC$, we denote by $\mathcal{M}_{\lambda}$ the set of all bounded measures $\mu$ on $M_0$ for which there is a sequence $(\tau_j)_{j=1}^{\infty}$ with $\tau_j \to \infty$ and a sequence $(v_j)_{j=1}^{\infty} \subset C^2(M_0)$ satisfying 
$$
\norm{(-\Delta_{g_0} - (\tau_j + i \lambda)^2) v_j}_{L^2(M_0)} = o(\tau_j), \quad \norm{v_j}_{L^2(M_0)} = O(1)
$$
as $j \to \infty$, such that in the weak topology of measures on $M_0$ one has 
$$
\lim_{j \to \infty} \abs{v_j}^2 \,dV_{g_0} = \mu
$$
where $dV_{g_0}$ is the volume form of $(M_0,g_0)$.
\end{definition}

\begin{thm} \label{claim3}
Let $(M,g)$ be a CTA manifold, and let $q_1, q_2 \in C(M)$. If $C_{g,q_1} = C_{g,q_2}$, then 
$$
\int_{M_0} \left[ \int_{-\infty}^{\infty}e^{-2i\lambda x_1} (c(q_1-q_2))(x_1,x') \,dx_1 \right] d\mu(x') = 0
$$
for any $\lambda \in \mC$ and any $\mu \in \mathcal{M}_{\lambda}$. Here $q_1-q_2$ is extended by zero to $\mR \times M_0$.
\end{thm}

The measures $\mu \in \mathcal{M}_{\lambda}$ are called \emph{semiclassical defect measures}, or \emph{quantum limits}, of the families of quasimodes $(v_{\tau+i\lambda})$. The properties of such measures are the central object of interest in the study of high frequency limits of eigenfunctions and in quantum ergodicity. In general, the dynamics of the geodesic flow of the underlying manifold $(M_0,g_0)$ will be visible in the semiclassical measures. These topics have a large literature, and we refer to \cite{JNT_survey}, \cite{Z_survey1}, \cite{Z_survey2} for surveys. However, our situation seems to be somewhat different from most of these works for the following three reasons:

\begin{enumerate}
\item
We only have access to limit measures in the base manifold $M_0$ instead of the more usual phase space measures in $T^* M_0$.
\item 
The measures $\mathcal{M}_{\lambda}$ are associated to a family of quasimodes in a manifold $(M_0,g_0)$ with boundary, but there is \emph{no boundary condition} imposed on the quasimodes. This leads to a certain amount of flexibility in our setting.
\item 
It is useful to consider measures for slightly complex frequencies $\tau + i\lambda$ where $\re(\lambda)$ is nonzero.
\end{enumerate}

Theorem \ref{claim2} will be obtained from Theorem \ref{claim3} by a rather direct construction of Gaussian beam quasimodes that concentrate on a given nontangential geodesic. This construction goes back at least to \cite{Arnold}, \cite{BabichBuldyrev}, \cite{BabichLazutkin}, \cite{CdV}, \cite{Hormander_gaussian} and has been developed further by many authors (often for hyperbolic equations), see for instance \cite{KKL}, \cite{Ralston}. In our case, we need the next result which follows by adapting the methods in the literature in a suitable way. The fact that the frequency is slightly complex leads to the attenuated geodesic ray transform with constant attenuation $-2\lambda$, but eventually analyticity will allow to make a reduction to the case $\lambda = 0$.

\begin{thm} \label{claim4}
Let $(M_0,g_0)$ be a compact oriented manifold with smooth boundary, let $\gamma: [0,L] \to M_0$ be a nontangential geodesic, and let $\lambda \in \mR$. For any $K > 0$ there is a family of functions $(v_s) \subset C^{\infty}(M_0)$, where $s = \tau+i\lambda$ and $\tau \geq 1$, such that 
$$
\norm{(-\Delta_{g_0} - s^2) v_s}_{L^2(M_0)} = O(\tau^{-K}), \quad \norm{v_s}_{L^2(M_0)} = O(1)
$$
as $\tau \to \infty$, and for any $\psi \in C(M_0)$ one has 
$$
\lim_{\tau \to \infty} \int_{M_0} \abs{v_s}^2 \psi \,dV_{g_0} = \int_0^L e^{-2\lambda t} \psi(\gamma(t)) \,dt.
$$
\end{thm}

We remark that a similar Gaussian beam quasimode construction was used to deal with partial data inverse problems in the paper \cite{KSa} which was in preparation simultaneously with this manuscript. It is an interesting question whether other quasimode constructions could be used to extract more information about the potentials via Theorem \ref{claim3}. In particular, the following question is of interest. (By Theorem \ref{claim4} we know that this question has a positive answer if $\lambda = 0$ for any $(M_0,g_0)$ in which the ray transform is injective; on the other hand having $\lambda \neq 0$ might help.)

\begin{question} \label{question_intro}
Let $(M_0,g_0)$ be a compact oriented manifold with smooth boundary, and let $\lambda \in \mR$. Under which conditions on $(M_0,g_0)$ is the set $\mathcal{M}_{\lambda}$ dense in the set of all bounded measures on $M_0$?
\end{question}

The previous results are all based on extensions of the complex geometrical optics method. In the final results of this paper, we will use a completely different approach and reduce the anisotropic Calder\'on problem to an inverse problem for the wave equation. To motivate this, note that the Laplace-Beltrami operator $\Delta_g$ in a product type manifold $(\mR \times M_0, g)$, where $g = e \oplus g_0$ and we now write $t$ for the Euclidean variable, has the form 
$$
\partial_t^2 + \Delta_{g_0}.
$$
By formally complexifying the $t$ variable by $t \mapsto it$ (Wick rotation), we arrive at the wave operator 
$$
\partial_t^2 - \Delta_{g_0}.
$$

Let us next describe a standard inverse problem for the wave equation. If $(M_0,g_0)$ is a compact oriented manifold with smooth boundary, if $q_0 \in C(M_0)$, and if $T > 0$, consider the initial-boundary value problem 
\begin{align*}
(\partial_t^2 - \Delta_{g_0} + q_0) u &= 0 \quad \text{in } (0,T) \times M_0, \\
u(0) = \partial_t u(0) &= 0, \\
u|_{(0,T) \times \partial M_0} &= f.
\end{align*}
This problem has a unique solution $u \in C^{\infty}((0,T) \times M_0)$ for any $f \in C^{\infty}_c((0,T) \times \partial M_0)$, and we can define the hyperbolic DN map 
$$
\Lambda_{g_0,q_0}^{Hyp}: C^{\infty}_c((0,T) \times \partial M_0) \to C^{\infty}((0,T) \times \partial M_0), \ \ f \mapsto \partial_{\nu} u|_{(0,T) \times \partial M_0}.
$$
The inverse problem is to determine the metric $g_0$ up to isometry and the potential $q_0$ from the knowledge of the DN map $\Lambda_{g_0,q_0}^{Hyp}$. This problem is closely related (and often equivalent) to an inverse boundary spectral problem \cite{KKL}, to a multidimensional Borg-Levinson theorem \cite{NSU}, and also to an inverse problem posed by Gel'fand \cite{Gelfand}. In this paper, the wave equation inverse problem will be called the \emph{Gel'fand problem}, although of course there are many other important problems due to Gel'fand.

The Gel'fand problem in the above formulation has a positive answer, under the natural necessary condition that $T > 2 r(M_0)$ where $r(M_0) = \sup \{ r > 0 \,;\, B(x,r) \subset M_0^{\text{int}} \text{ for some } x \in M_0 \}$ is the time needed to fill in the manifold by waves from the boundary. This follows from the boundary control method introduced by Belishev \cite{Belishev} and later developed by several authors; we refer to the book \cite{KKL} for further details. The boundary control method is based on three components:

\begin{enumerate}
\item
Integration by parts (\emph{Blagovestchenskii identity}): recover inner products of solutions at a fixed time from the hyperbolic DN map.
\item 
Approximate controllability based on the unique continuation theorem of Tataru \cite{Tataru}: solutions $u(t_0,\,\cdot\,)$ are $L^2$ dense in the appropriate domain of influence.
\item 
Recovering the coefficients: this uses a boundary distance representation of $(M_0,g_0)$ together with projectors to domains of influence and special solutions such as Gaussian beams.
\end{enumerate}

An elliptic analogue of the Gel'fand problem is given by the following version of the anisotropic Calder\'on problem. Let $(M_0,g_0)$ be a compact oriented manifold with smooth boundary, let $q_0 \in C^{\infty}(M_0)$, and let $T = \mR \times M_0$ be an infinite cylinder equipped with the metric $g = e \oplus g_0$. Write $(t,x)$ for the coordinates in $\mR \times M_0$. Let also $\text{Spec}(-\Delta_{g_0}+q_0) = \{ \lambda_j \}_{j=1}^{\infty}$ where $\lambda_1 \leq \lambda_2 \leq \ldots$ is the set of Dirichlet eigenvalues of $-\Delta_{g_0}+q_0$ in $(M_0,g_0)$. Consider the Schr\"odinger equation in $T$, 
$$
(-\partial_t^2 - \Delta_{g_0} + q_0 - \lambda)u = 0 \ \ \text{in } T, \qquad u|_{\partial T} = f.
$$

Let us first make the assumption that $\lambda \in \mC \setminus [\lambda_1,\infty)$, that is, $\lambda$ is outside the continuous spectrum of $-\Delta_g + q_0$ in $T$. Then for any $f \in C^{\infty}_c(\partial T)$ the above equation has a unique solution $C^{\infty}(T) \cap H^1(T)$, and there is a linear DN map 
$$
\Lambda_{g_0,q_0}^{Ell}(\lambda): C^{\infty}_c(\partial T) \to C^{\infty}(\partial T), \ \ f \mapsto \partial_{\nu} u|_{\partial T}.
$$

The next result shows that one can reconstruct the isometry class of an unknown manifold $(M_0,g_0)$ and also a potential $q_0$ from the knowledge of $\partial M_0$ and the DN map $\Lambda_{g_0,q_0}^{Ell}(\lambda)$.

\begin{thm} \label{thm_calderon_cylinder_reconstruction}
Given the data $(\partial T, \Lambda_{g_0,q_0}^{Ell}(\lambda))$ for a fixed $\lambda \in \mC \setminus [\lambda_1,\infty)$, where $\partial T = \mR \times \partial M_0$ and $\Lambda_{g_0,q_0}^{Ell}(\lambda): C^{\infty}_c(\partial T) \to C^{\infty}(\partial T)$ corresponds to the Schr\"odinger operator $-\Delta_g+q_0$ on $T$, one can reconstruct the potential $q_0$ and a Riemannian manifold $(\widehat{M}_0,\widehat{g}_0)$ isometric to $(M_0,g_0)$.
\end{thm}

We obtain a uniqueness result as a consequence ($\widetilde{\lambda}_1$ is the first Dirichlet eigenvalue of $-\Delta_{\widetilde{g}_0}+\widetilde{q}_0$ in $M_0$):

\begin{thm} \label{thm_calderon_cylinder_uniqueness}
Let $(M_0,g_0)$ and $(M_0,\widetilde{g}_0)$ be two compact manifolds with boundary $\partial M_0$, and let $q_0, \widetilde{q}_0 \in C^{\infty}(M_0)$. If 
$$
\Lambda_{g_0,q_0}^{Ell}(\lambda) = \Lambda_{\widetilde{g}_0,\widetilde{q}_0}^{Ell}(\lambda) \ \ \text{for some } \lambda \in \mC \setminus ([\lambda_1,\infty) \cup [\widetilde{\lambda}_1,\infty)),
$$
then $\widetilde{g}_0 = \psi_0^* g_0$ for some diffeomorphism $\psi_0: M_0 \to M_0$ with $\psi_0|_{\partial M_0} = \text{Id}$, and also $\widetilde{q}_0 = \psi_0^* q_0$.
\end{thm}

Next we consider the case where the spectral parameter is in the continuous spectrum but not in the set of thresholds (that is, $\lambda \in [\lambda_1,\infty) \setminus \text{Spec}(-\Delta_{g_0}+q_0)$). In that case one needs a radiation condition to have a well defined DN map, and one obtains the following result for the inverse problem. We refer to Section \ref{sec_bcmethod_second} for the details.
  
\begin{thm} \label{thm_calderon_cylinder_reconstruction2}
Given the data $(\partial T, \Lambda_{g_0,q_0}^{T}(\lambda))$ for a fixed $\lambda \in [\lambda_1, \infty) \setminus \{ \lambda_1,\lambda_2,\ldots \}$, where $\partial T = \mR \times \partial M_0$ and $\Lambda_{g_0,q_0}^{T}(\lambda): C^{\infty}_c(\partial T) \to C^{\infty}(\partial T)$ corresponds to the Schr\"odinger operator $-\Delta+q_0$ on $T$, one can reconstruct the potential $q_0$ and a Riemannian manifold $(\widehat{M}_0,\widehat{g}_0)$ isometric to $(M_0,g_0)$.
\end{thm}

As mentioned above, the proof involves a reduction from the elliptic DN map to the hyperbolic DN map and the boundary control method. We also use the elliptic DN map on the transversal manifold, defined for $\lambda$ outside $\text{Spec}(-\Delta_{g_0}+q_0)$ by 
$$
\Lambda_{g_0,q_0}^{Tr}(\lambda): v|_{\partial M_0} \mapsto \partial_{\nu} v|_{\partial M_0}, \quad (-\Delta_{g_0}+q_0-\lambda) v = 0 \text{ in } M_0.
$$
The argument proceeds roughly as follows:
\begin{enumerate}
\item 
Extend $\Lambda_{g_0,q_0}^{Ell}(\lambda)$ to act on weighted Sobolev spaces on $\partial T$.
\item 
If $k > 0$, obtain $\Lambda_{g_0,q_0}^{Tr}(\lambda-k^2)$ for any $h \in C^{\infty}(\partial M_0)$ via 
$$
\Lambda_{g_0,q_0}^{Tr}(\lambda-k^2) h = e^{-ikt} \Lambda_{g_0,q_0}^{Ell}(\lambda) (e^{ikt} h).
$$
\item 
Recover $\Lambda_{g_0,q_0}^{Tr}(\mu)$ for $\mu \in \mC$ from $\{ \Lambda_{g_0,q_0}^{Tr}(\lambda-k^2) \}_{k > 0}$ by meromorphic continuation.
\item 
Recover $\Lambda_{g_0,q_0}^{Hyp}$ from $\{ \Lambda_{g_0,q_0}^{Tr}(\mu) \}_{\mu \in \mC}$ by Laplace transform in time.
\item 
Use the boundary control method to determine $(M_0,g_0)$ up to isometry and $q_0$ from $\Lambda_{g_0,q_0}^{Hyp}$.
\end{enumerate}

It was proved in \cite{KKLM} that knowing the transversal DN maps $\{ \Lambda_{g_0,q_0}^{Tr}(\mu) \}_{\mu \in \mC}$ is equivalent to knowing the DN map for the following equations:
\begin{itemize}
\item 
Wave equation $(\partial_t^2 - \Delta_{g_0} + q_0) u = 0$ in $(0,\infty) \times M_0$,
\item 
Heat equation $(\partial_t - \Delta_{g_0} + q_0) u = 0$ in $(0,\infty) \times M_0$,
\item 
Schr\"odinger equation $(i\partial_t - \Delta_{g_0} + q_0) u = 0$ in $(0,\infty) \times M_0$.
\end{itemize}
Our results show that the elliptic equation $(-\partial_t^2 - \Delta_{g_0} + q_0) u = 0$ in $\mR \times M_0$ can be added to this list.

Note that Theorems \ref{thm_calderon_cylinder_reconstruction} and \ref{thm_calderon_cylinder_uniqueness} are valid for arbitrary transversal manifolds $M_0$ without any restriction on the geometry, and they allow to recover both the transversal metric and the potential from the elliptic DN map. They are also the first uniqueness results for the Calder\'on problem that we are aware of which employ control theory methods (in particular approximate controllability based on unique continuation for the wave equation). At the moment we can only show these results by going through the wave equation. It would be interesting to understand if there is a proof that would work with the elliptic equation directly.

However, there is a severe restriction: the potential $q_0$ has to be independent of the $t$ variable, unlike in Theorems \ref{claim1}--\ref{claim3} where the scalar coefficient may depend on the Euclidean variable. In fact, the analogue of Theorem \ref{thm_calderon_cylinder_uniqueness} on a fixed compact manifold $(M,g) \subset \subset (\mR \times M_0, e \oplus g_0)$ with two potentials independent of the $t$ variable can easily be reduced to standard boundary determination results \cite[Section 8]{DKSaU}. Of course, in the infinite cylinder $T$ boundary determination is not so helpful and we use a reduction to the wave equation instead.

The Wick rotation $t \mapsto it$ suggests that the potential $q_0$ should indeed be independent of the $t$ variable, or at least real analytic in $t$, for this reduction to the wave equation to work. The boundary control method for the wave equation also requires the coefficients to be independent of the time variable, although a variant of this method due to Eskin \cite{Eskin_BCanalytic} allows lower order coefficients that are real analytic in time. The Gel'fand problem for time-dependent coefficients is interesting in its own right; see \cite{RammSjostrand}, \cite{Salazar}, \cite{Stefanov} for some results when the background metric is Euclidean.

This paper is structured as follows. Section \ref{sec_intro} is the introduction, and Section \ref{sec_cgo} gives the construction of complex geometrical optics solutions based on quasimodes and proves Theorem \ref{claim3}. Section \ref{sec_quasimodes} contains a direct construction of Gaussian beam quasimodes and the proofs of Theorems \ref{claim1}, \ref{claim2} and \ref{claim4}. In Section \ref{sec_microlocal} we give an alternative construction of Gaussian beam quasimodes based on a microlocal reduction via Fourier integral operators. The Calder\'on problem in an infinite cylinder is considered in the last two sections. Section \ref{sec_bcmethod_first} discusses the case where the spectral parameter is outside the continuous spectrum and gives the proofs of Theorems \ref{thm_calderon_cylinder_reconstruction} and \ref{thm_calderon_cylinder_uniqueness}, and Section \ref{sec_bcmethod_second} extends these results to the case where the spectral parameter may be in the continuous spectrum but not in the set of thresholds.

\bigskip

\noindent {\bf Acknowledgements.}
Y.K.~is partly supported by EPSRC, M.L.~and M.S.~are supported in part by the \mbox{Academy} of Finland (Finnish Centre of Excellence in Inverse Problems Research), and M.S.~is also supported by an ERC Starting Grant. The authors would like to thank the organizers of the MSRI program on Inverse Problems and Applications in 2010, the Isaac Newton Institute program on Inverse Problems in 2011, the Fields Institute program on Geometry in Inverse Problems in 2012, and the Institut Mittag-Leffler program on Inverse Problems in 2013 where part of this work was carried out. D.DSF.~would like to acknowledge the hospitality of the University of Jyv\"askyl\"a.

\section{Complex geometrical optics} \label{sec_cgo}

In this section we explain the construction of complex geometrical optics solutions based on quasimodes in $(M_0,g_0)$ and use this construction to prove Theorem \ref{claim3}. The argument is close to \cite[Section 5]{DKSaU}.

We will assume that $(M,g)$ is CTA with $(M,g) \subset \subset (\mR \times M_0,g)$ where $g = c(e \oplus g_0)$, and $(M_0,g_0)$ is a compact $(n-1)$-dimensional manifold with boundary. Let also $q \in L^{\infty}(M)$. We first note the identity 
$$
c^{\frac{n+2}{4}}(-\Delta_g + q)(c^{-\frac{n-2}{4}} \widetilde{u}) = (-\Delta_{\widetilde{g}} + \widetilde{q})\widetilde{u}
$$
where 
$$
\widetilde{g} = e \oplus g_0, \quad \widetilde{q} = c(q - c^{\frac{n-2}{4}} \Delta_g (c^{-\frac{n-2}{4}})).
$$
This shows that it is enough to construct solutions to $(-\Delta_{\widetilde{g}} + \widetilde{q})\widetilde{u} = 0$.

Writing $x = (x_1,x')$ for coordinates in $\mR \times M_0$, the function $\varphi(x) = x_1$ is a limiting Carleman weight in a neighborhood of $M$ \cite{DKSaU}. In particular, we have the following solvability result which follows from \cite[Section 4]{DKSaU} (see also \cite[Section 4]{KSaU} where one obtains $H^2$ solutions).

\begin{prop} \label{prop_carleman_estimates}
Let $\widetilde{q} \in L^{\infty}(M)$. There exists $\tau_0 \geq 1$ such that whenever $\abs{\tau} \geq \tau_0$, then for any $f \in L^2(M)$ the equation 
$$
e^{\tau x_1} (-\Delta_{\widetilde{g}} + \widetilde{q}) e^{-\tau x_1} r = f \quad \text{in } M
$$
has a solution $r \in H^1(M)$ satisfying the estimates 
$$
\norm{r}_{H^{\alpha}(M)} \leq C \abs{\tau}^{\alpha-1} \norm{f}_{L^2(M)}, \quad 0 \leq \alpha \leq 1.
$$
\end{prop}

Consider complex frequencies 
$$
s = \tau + i\lambda, \quad \text{$\tau$ real with $\abs{\tau}$ large, $\lambda$ complex and fixed}.
$$
We are interested in finding complex geometrical optics solutions to the equation $(-\Delta_{\widetilde{g}} + \widetilde{q})u = 0$ in $M$, having the form 
$$
u = e^{-s x_1}(v+r).
$$
Here $v = v_s$ will be an amplitude type term, and $r = r_s$ is a correction term with $\norm{r}_{L^2(M)} \to 0$ as $\abs{\tau} \to \infty$. Further, we require certain asymptotic properties of $v_s$ as $\abs{\tau} \to \infty$ when $\lambda$ is kept fixed.

A function $u$ of the above type is a solution provided that 
$$
e^{\tau x_1}(-\Delta_{\widetilde{g}} + \widetilde{q}) e^{-\tau x_1}(e^{-i\lambda x_1} r) = f
$$
where the right hand side is 
$$
f = -e^{-i\lambda x_1} e^{s x_1}(-\Delta_{\widetilde{g}} + \widetilde{q})e^{-s x_1} v.
$$
The point is to choose $v$ so that $\norm{f}_{L^2(M)}$ does not grow when $\abs{\tau} \to \infty$, and to choose $r$ so that $e^{-i\lambda x_1} r$ is the solution given by Proposition \ref{prop_carleman_estimates}.

At this point we use the product structure on $(\mR \times M_0, \widetilde{g})$ where $\widetilde{g} = e \oplus g_0$, which implies that $\Delta_{\widetilde{g}} = \partial_1^2 + \Delta_{g_0}$. Consequently 
$$
e^{s x_1}(-\Delta_{\widetilde{g}} + \widetilde{q})e^{-s x_1} v = (-\partial_1^2 + 2s \partial_1 - s^2 - \Delta_{g_0} + \widetilde{q}) v.
$$
This expression simplifies if we choose $v$ independent of $x_1$, that is, $v = v(x')$, and in this case 
$$
f = -e^{i\lambda x_1}(-\Delta_{g_0} - s^2 + \widetilde{q})v.
$$
Now $\norm{f}_{L^2(M)}$ will not be too large with respect to $\abs{\tau}$ if $v = v_s(x')$ is a \emph{quasimode} or an \emph{approximate eigenfunction} in the transversal manifold $(M_0,g_0)$, in the sense that 
$$
\norm{(-\Delta_{g_0} - s^2) v_s}_{L^2(M_0)} = o(\abs{\tau}), \quad \norm{v_s}_{L^2(M_0)} = O(1)
$$
as $\abs{\tau} \to \infty$.

The following result describes the complex geometrical optics solutions.

\begin{prop} \label{prop_cgo_solutions}
Let $q \in L^{\infty}(M)$, let $\tau_0$ be sufficiently large, and let $\lambda$ be a fixed real number. Suppose that $\{v_s \,;\, s = \tau+i\lambda, \abs{\tau} \geq \tau_0 \}$ is a family of functions in $L^2(M_0)$ satisfying 
$$
\norm{(-\Delta_{g_0} - s^2) v_s}_{L^2(M_0)} = o(\abs{\tau}), \quad \norm{v_s}_{L^2(M_0)} = O(1)
$$
as $\abs{\tau} \to \infty$. Then for any $\tau$ with $\abs{\tau} \geq \tau_0$ there is a solution $u \in H^1(M)$ of $(-\Delta_g + q)u = 0$ in $M$ having the form 
$$
u = e^{-s x_1}c^{-\frac{n-2}{4}}(v_s + r_s)
$$
where $\norm{r_s}_{L^2(M)} = o(1)$ as $\abs{\tau} \to \infty$.
\end{prop}
\begin{proof}
We first produce a solution of the equation $(-\Delta_{\widetilde{g}} + \widetilde{q})\widetilde{u} = 0$ having the form $\widetilde{u} = e^{-s x_1}(v_s+r_s)$ as in the preceding discussion, and then define $u = c^{-\frac{n-2}{4}} \widetilde{u}$ to obtain a corresponding solution of $(-\Delta_g+q)u = 0$.
\end{proof}

The next result is slightly more general than Theorem \ref{claim3}.

\begin{prop} \label{prop_cgo_density_general}
Let $(M,g) \subset \subset (\mR \times M_0,g)$ be a CTA manifold, where $g = c(e \oplus g_0)$, and let $q_1, q_2 \in C(M)$. Let $\lambda_1, \lambda_2 \in \mC$, let $(\tau_j)_{j=1}^{\infty}$ be sequence of positive numbers with $\tau_j \to \infty$, and let 
$$
s_j = \tau_j + i\lambda_1, \quad t_j = \tau_j + i\bar{\lambda}_2.
$$
Suppose that $(v_{s_j}), (w_{t_j}) \subset C^2(M_0)$ are sequences satisfying 
\begin{gather*}
\norm{(-\Delta_{g_0} - s_j^2)v_{s_j}}_{L^2(M_0)} = o(\tau_j), \quad \norm{v_{s_j}}_{L^2(M_0)} = O(1), \\
\norm{(-\Delta_{g_0} - t_j^2)w_{t_j}}_{L^2(M_0)} = o(\tau_j), \quad \norm{w_{t_j}}_{L^2(M_0)} = O(1)
\end{gather*}
as $j \to \infty$, and in the weak topology of measures on $M_0$, 
$$
\lim_{j \to \infty} v_{s_j} \overline{w_{t_j}} \,dV_{g_0} =  \mu_{\lambda_1,\lambda_2}
$$
for some bounded measure $\mu_{\lambda_1,\lambda_2}$ on $M_0$. If 
$$
C_{g,q_1} = C_{g,q_2},
$$
then 
$$
\int_{M_0} \left[ \int_{-\infty}^{\infty} e^{-i(\lambda_1+\lambda_2) x_1} (c(q_1-q_2))(x_1,x') \,dx_1 \right] d\mu_{\lambda_1,\lambda_2}(x') = 0.
$$
Here $q_1-q_2$ is extended by zero to $\mR \times M_0$.
\end{prop}
\begin{proof}
We use Proposition \ref{prop_cgo_solutions} to find solutions of $(-\Delta_g+q_1)u_{s_j} = 0$ and $(-\Delta_g + \overline{q_2}) u_{t_j} = 0$, of the form 
\begin{gather*}
u_{s_j} = e^{-s_j x_1} c^{-\frac{n-2}{4}} (v_{s_j} + r_{s_j}), \\
u_{t_j} = e^{t_j x_1} c^{-\frac{n-2}{4}} (w_{t_j} + r_{t_j})
\end{gather*}
where $\norm{r_{s_j}}_{L^2(M)}, \norm{r_{t_j}}_{L^2(M)} = o(1)$ as $j \to \infty$. Note that $\bar{u}_{t_j}$ solves the equation $(-\Delta+q_2) \bar{u}_{t_j} = 0$ in $M$.

Next follows the usual integration by parts: we have 
\begin{align*}
\int_M (q_1-q_2) u_{s_j} \bar{u}_{t_j} \,dV &= \int_M \left[ (\Delta_g u_{s_j}) \bar{u}_{t_j} - u_{s_j} (\Delta_g \bar{u}_{t_j}) \right] \,dV \\
 &= \int_{\partial M} \left[ (\partial_{\nu} u_{s_j}) \bar{u}_{t_j} - u_{s_j} (\partial_{\nu} \bar{u}_{t_j})\right] \,dS
\end{align*}
where the normal derivatives of the $H^1$ solutions are interpreted in the weak sense as elements in $H^{-1/2}(\partial M)$. Using the condition $C_{g,q_1} = C_{g,q_2}$, there is some $\widetilde{u} \in H^1(M)$ with $(-\Delta_g+q_2)\widetilde{u} = 0$ in $M$ and 
$$
\widetilde{u}|_{\partial M} = u_{s_j}|_{\partial M}, \quad \partial_{\nu} \widetilde{u}|_{\partial M} = \partial_{\nu} u_{s_j}|_{\partial M}.
$$
This shows that 
\begin{align*}
\int_M (q_1-q_2) u_{s_j} \bar{u}_{t_j} \,dV &= \int_{\partial M} \left[ (\partial_{\nu} \widetilde{u}) \bar{u}_{t_j} - \widetilde{u} (\partial_{\nu} \bar{u}_{t_j}) \right] \,dS \\
 &= \int_M \left[ (\Delta_g \widetilde{u}) \bar{u}_{t_j} - \widetilde{u} (\Delta_g \bar{u}_{t_j}) \right] \,dV \\
 &= \int_M (q_2 - q_2) \widetilde{u} \bar{u}_{t_j} \,dV = 0.
\end{align*}

Substituting the forms of the solutions $u_{s_j}$ and $\bar{u}_{t_j}$ in the last identity, we see that 
$$
\int_M (q_1-q_2) e^{-i(\lambda_1+\lambda_2) x_1} c^{-\frac{n-2}{2}} v_{s_j} \overline{w_{t_j}} \,dV = o(1) \quad \text{as $j \to \infty$},
$$
using the norm estimate for the correction terms $r_{s_j}$ and $r_{t_j}$ and the $L^2$ estimates for $v_{s_j}$ and $w_{t_j}$. We now extend $q_1-q_2$ by zero to $\mR \times M_0$ and note that $dV_g(x) = c^{n/2} \,dx_1 \,dV_{g_0}(x')$. Then, taking the limit as $j \to \infty$ and using the assumption that $v_{s_j} \overline{w_{t_j}}$ converges in the weak topology of measures, we obtain 
$$
\int_{M_0} \left[ \int_{-\infty}^{\infty} e^{-i(\lambda_1+\lambda_2) x_1} (c(q_1-q_2))(x_1,x') \,dx_1 \right] d\mu_{\lambda_1,\lambda_2}(x') = 0.
$$
To be precise, we would like that the expression in brackets is a continuous function with respect to $x'$ in $M_0$ in order to take the limit. However, the condition $C_{g,q_1} = C_{g,q_2}$ implies by boundary determination that $q_1|_{\partial M} = q_2|_{\partial M}$, and thus the zero extension of $q_1-q_2$ is in fact a continuous compactly supported function in $\mR \times M_0$. The boundary determination result is essentially contained in \cite[Proposition A.1]{GT} for the case $n=2$, and a similar argument works also for $n \geq 3$ (see \cite[Section 8]{DKSaU} for the case of DN maps with smooth $q_1$ and $q_2$).
\end{proof}

\begin{proof}[Proof of Theorem \ref{claim3}]
This follows from Proposition \ref{prop_cgo_density_general} by taking $\lambda_1 = \lambda_2$ real and taking $v_{s_j} = w_{s_j}$.
\end{proof}

At this point it is useful to compare the solutions in Proposition \ref{prop_cgo_solutions} to the ones appearing in \cite[Section 5]{DKSaU}, where the additional assumption that $(M_0,g_0)$ is simple was imposed. The complex geometrical optics solutions in \cite{DKSaU}, satisfying $(-\Delta_g + q)u = 0$ in $M$, have the form 
$$
u = e^{-\tau x_1} (e^{-i\tau \psi} a + r).
$$
Here $\psi$ is a real function chosen as a solution of an eikonal equation, and the amplitude $a$ solves a complex transport equation in $M$. Since $(M_0,g_0)$ is simple these equations can be solved globally in $M$, and in fact $\psi$ only depends on $x'$. Then $e^{-i\tau \psi} a$ satisfies 
$$
e^{\tau x_1}(-\Delta_g)e^{-\tau x_1}(e^{-i\tau \psi} a) = O_{L^2(M)}(1)
$$
as $\tau \to \infty$. If $a$ would be independent of $x_1$, then $e^{-i\tau \psi(x')} a(x')$ would be an approximate eigenfunction in $M_0$ in the sense that 
$$
(-\Delta_{g_0} - \tau^2)(e^{-i\tau \psi} a) = O_{L^2(M_0)}(1).
$$
However, such functions are not quite sufficient to prove uniqueness results for the inverse problem. In \cite{DKSaU} one instead employed amplitudes of the form $a(x_1,x') = e^{-i\lambda x_1} \widetilde{a}(x')$ which allow to exploit the Fourier transform in $x_1$.

There are two differences between Proposition \ref{prop_cgo_solutions} and the construction in \cite{DKSaU}, although the two are very closely related. The first one is that we use large complex frequencies $s = \tau + i\lambda$ instead of large real frequencies $\tau$, which amounts to incorporating the factor $e^{-i\lambda x_1}$ from the amplitude $a$ as part the complex frequency (thus making it possible to use the Fourier transform in $x_1$). The second difference is roughly that instead of using approximate eigenfunctions $e^{-i\tau \psi(x')} a(x')$ with real frequency, we consider more general approximate eigenfunctions $v_s(x')$ with slightly complex frequency. This approach loses some generality since $v_s$ is not allowed to depend on $x_1$, but has the benefit that one can use much more general approximate eigenfunctions $v_s(x')$ than those of the form $e^{-i\tau \psi(x')} a(x')$ obtained from a global WKB construction on $M_0$.



\section{Gaussian beam quasimodes} \label{sec_quasimodes}

We will now give the Gaussian beam construction of approximate eigenfunctions, or quasimodes, with desirable concentration properties. In fact, these quasimodes will concentrate near a geodesic in the high frequency limit. On a compact manifold without boundary, it is well known that one can find quasimodes concentrating near a stable closed geodesic for large real frequencies. We refer to \cite[Section 10]{Z_survey1} and the references therein.

The setup here is more flexible since there are no boundary conditions or global conditions on a closed manifold required of the family $\{v_s\}$. Therefore, a construction of local nature is sufficient. We will give a direct argument analogous to the construction of Gaussian beams, which are approximate solutions of the wave equation localized near a geodesic \cite{KKL}. The fact that we need approximate eigenfunctions with slightly complex frequencies instead of real ones will not present any complications. A version of this construction that also takes into account possible reflections is given in \cite{KSa}.

For most of this section we will write $(M,g)$ for the transversal manifold instead of $(M_0,g_0)$ in order to simplify notation. Let $(M,g)$ be an $m$-dimensional compact oriented manifold with smooth boundary (thus $m=n-1 \geq 2$). Recall that a unit speed geodesic $\gamma: [0,L] \to M$ is called nontangential if $\dot{\gamma}(0), \dot{\gamma}(L)$ are nontangential vectors on $\partial M$ and $\gamma(t) \in M^{\text{int}}$ for $0 < t < L$. Theorem \ref{claim4} is the following statement.

\begin{prop} \label{prop_gaussianbeam_quasimode}
Let $\gamma: [0,L] \to M$ be a nontangential geodesic, and let $\lambda \in \mR$. For any $K > 0$ there is a family of functions $(v_s) \subset C^{\infty}(M)$, where $s = \tau+i\lambda$ and $\tau \geq 1$, such that 
$$
\norm{(-\Delta_{g} - s^2) v_s}_{L^2(M)} = O(\tau^{-K}), \quad \norm{v_s}_{L^2(M_0)} = O(1)
$$
as $\tau \to \infty$â and for any $\psi \in C(M)$ one has 
$$
\lim_{\tau \to \infty} \int_{M} \abs{v_s}^2 \psi \,dV_{g} = \int_0^L e^{-2\lambda t} \psi(\gamma(t)) \,dt.
$$
\end{prop}

In the case where $(M,g)$ is simple, the method in \cite{DKSaU} (although it was not written exactly in this way) reduces to using approximate eigenfunctions of the above type to recover attenuated geodesic ray transforms of desired quantities. In fact, a version of Proposition \ref{prop_gaussianbeam_quasimode} on simple manifolds follows easily from the methods in \cite{DKSaU}.

\begin{prop}
Let $(M,g)$ be simple, let $\lambda$ be a fixed real number, and let $\gamma:[0,L] \to M$ be a nontangential geodesic. For any $0 < \alpha < 1$ there is a family $\{v_s \,;\,  s=\tau+i\lambda, \ \tau \geq 1\}$ in $C^{\infty}(M)$ such that 
$$
\norm{(-\Delta_{g} - s^2) v_s}_{L^2(M)} = O(\tau^{\alpha}), \quad \norm{v_s}_{L^2(M)} = O(1)
$$
as $\tau \to \infty$, and for any $\psi \in C(M)$ 
$$
\int_{M} \abs{v_{\tau+i\lambda}}^2 \psi \,dV_{g} \to \int_0^L e^{-2\lambda t} \psi(\gamma(t)) \,dt \quad \text{as $\tau \to \infty$}.
$$
\end{prop}
\begin{proof}
One first embeds $(M,g)$ in a slightly larger simple manifold $(D,g)$ and considers polar normal coordinates $(r,\theta)$ centered at a point $\omega \in D \smallsetminus M$. There exist $\omega \in D \smallsetminus M$ and $\theta_0 \in S^{m-1}$ so that $\gamma$ is part of the geodesic $r \mapsto (r,\theta_0)$ in $D$ (any nontangential geodesic in $M$ arises in this way for some $\omega$ and $\theta_0$). By using a WKB ansatz and choosing suitable solutions of the eikonal and transport equations as in \cite[Section 5]{DKSaU}, the quasimodes at frequency $s = \tau+i\lambda$ can be chosen as 
$$
v_s(r,\theta) = e^{isr} \abs{g(r,\theta)}^{-1/4} b_{\tau}(\theta)
$$
where $b_{\tau} \in C^{\infty}(S^{m-1})$ is an approximation of the delta function so that 
\begin{gather*}
\norm{b_{\tau}}_{L^2(S^{m-1})} = 1, \quad \norm{b_{\tau}}_{W^{2,\infty}(S^{m-1})} = O(\tau^{\alpha}), \\
\abs{b_{\tau}}^2 \,dS \to \delta_{\theta_0} \text{ weakly as $\tau \to \infty$}.
\end{gather*}
A direct computation shows the required norm bounds, and we have for any $\psi \in C^{\infty}_c(M^{\text{int}})$ 
$$
\int_{M} \abs{v_{\tau+i\lambda}}^2 \psi \,dV_{g} \to \int_0^L e^{-2\lambda r} \psi(r,\theta_0) \,dr \quad \text{as } \tau \to \infty.
$$
\end{proof}

We now move to the proof of Proposition \ref{prop_gaussianbeam_quasimode}. The main difference to the case where $(M,g)$ is simple is that the quasimodes can not be constructed using the WKB ansatz by solving eikonal and transport equations globally in $M$, due to the presence of conjugate points. Instead, we follow the construction of Gaussian beams: the eikonal and transport equations are only solved to high order on the geodesic, and we employ a complex phase function with Gaussian decay away from the geodesic. The phase function will be obtained by solving a matrix Riccati equation. It will be computationally convenient to use Fermi coordinates, since these are globally defined near a geodesic (modulo possible self-intersections).

We first record a few elementary lemmas (for proofs see \cite[Section 7]{KSa}).

\begin{lemma} \label{lemma_selfintersection_elementary}
Let $(\widehat{M},g)$ be a closed manifold, and let $\gamma: (a,b) \to \widehat{M}$ be a unit speed geodesic segment having no loops. Then there are only finitely many times $t \in (a,b)$ for which $\gamma$ intersects itself at $\gamma(t)$.
\end{lemma}

\begin{lemma} \label{lemma_diffeomorphism_elementary}
Let $F$ be a $C^{\infty}$ map from a neighborhood of $(a,b) \times \{0\}$ in $\mR^n$ into a smooth manifold such that $F|_{(a,b) \times \{0\}}$ is injective and $DF(t,0)$ is invertible for $t \in (a,b)$. If $[a_0,b_0]$ is a closed subinterval of $(a,b)$, then $F$ is a $C^{\infty}$ diffeomorphism in some neighborhood of $[a_0,b_0] \times \{0\}$ in $\mR^n$.
\end{lemma}

The next lemma gives a system of Fermi coordinates near a geodesic that will be useful for the construction of Gaussian beam quasimodes. If the geodesic self-intersects, one needs several coordinate neighborhoods.

\begin{lemma} \label{lemma_quasimode_coordinates}
Let $(\widehat{M},g)$ be a compact manifold without boundary, and assume that $\gamma: (a,b) \to \widehat{M}$ is a unit speed geodesic segment with no loops. Given a closed subinterval $[a_0,b_0]$ of $(a,b)$ such that $\gamma|_{[a_0,b_0]}$ self-intersects only at times $t_j$ with $a_0 < t_1 < \ldots < t_N < b_0$ (set $t_0 = a_0$ and $t_{N+1} = b_0$), there is an open cover $\{(U_j,\varphi_j)\}_{j=0}^{N+1}$ of $\gamma([a_0,b_0])$ consisting of coordinate neighborhoods having the following properties:
\begin{enumerate}
\item 
$\varphi_j(U_j) = I_j \times B$ where $I_j$ are open intervals and $B = B(0,\delta')$ is an open ball in $\mR^{n-1}$ where $\delta'$ can be taken arbitrarily small,
\item 
$\varphi_j(\gamma(t)) = (t,0)$ for $t \in I_j$,
\item 
$t_j$ only belongs to $I_j$ and $\overline{I}_j \cap \overline{I}_k = \emptyset$ unless $\abs{j-k} \leq 1$,
\item 
$\varphi_j = \varphi_k$ on $\varphi_j^{-1}((I_j \cap I_k) \times B)$.
\end{enumerate}
Further, the metric in these coordinates satisfies $g^{jk}|_{\gamma(t)} = \delta^{jk}$, $\partial_i g^{jk}|_{\gamma(t)} = 0$.
\end{lemma}
\begin{proof}
The proof is based on Fermi coordinates. Choose $\{ v_2, \ldots, v_m \}$ in $T_{\gamma(a_0)} \widehat{M}$ such that $\{ \dot{\gamma}(a_0), v_2, \ldots, v_m \}$ is an orthonormal basis of $T_{\gamma(a_0)} \widehat{M}$. Let $E_{\alpha}(t)$ be the parallel transport of $v_{\alpha}$ along the geodesic $\gamma$. Since $\dot{\gamma}(t)$ is also parallel along $\gamma$, the set $\{ \dot{\gamma}(t), E_2(t), \ldots, E_m(t) \}$ is an orthonormal basis of $T_{\gamma(t)} \widehat{M}$ for $t \in (a,b)$.

Define the function 
$$
F: (a,b) \times \mR^{n-1} \to \widehat{M}, \ \ F(t,y) = \exp_{\gamma(t)}(y^{\alpha} E_{\alpha}(t)).
$$
Here $\exp$ is the exponential map in $(\widehat{M},g)$ and $\alpha$, $\beta$ run from $2$ to $m$. Then $F(t,0) = \gamma(t)$ and (with $e_{\alpha}$ the $\alpha$th coordinate vector) 
$$
\frac{\partial}{\partial s} F(t,s e_{\alpha})\big|_{s=0} = E_{\alpha}(t), \quad \frac{\partial}{\partial t} F(t,0) = \dot{\gamma}(t).
$$
Thus $F$ is a $C^{\infty}$ map near $(a,b) \times \{0\}$ such that $DF(t,0)$ is invertible for $t \in (a,b)$.

In the case where $\gamma$ does not self-intersect, $F|_{(a,b) \times \{0\}}$ is injective and Lemma \ref{lemma_diffeomorphism_elementary} implies the existence of a single coordinate neighborhood of $\gamma([a_0,b_0])$ so that (1) and (2) are satisfied (then (3) and (4) are void). In the general case, by Lemma \ref{lemma_selfintersection_elementary} the geodesic segment $\gamma|_{[a_0,b_0]}$ only self-intersects at finitely many times $t_j$ with $a_0 < t_1 < \ldots < t_N < b_0$ . For some sufficiently small $\delta$, $\gamma$ is injective on the intervals $(a,t_1-\delta)$, $(t_1-2\delta,t_2-\delta)$, \ldots, $(t_N-2\delta,b)$ and each interval intersects at most two of the others. Restricting the map $F$ above to suitable neighborhoods corresponding to these intervals (or slightly smaller ones) and using Lemma \ref{lemma_diffeomorphism_elementary}, we obtain the required coordinate charts with $\varphi_j = F^{-1}|_{U_j}.$

It remains to check the form of the metric in the coordinates $x$ where $x^1 = t$ and $x^{\alpha} = y^{\alpha}$ with $\alpha = 2, \ldots, m$. Since the set $\{ \dot{\gamma}(t), E_2(t), \ldots, E_m(t) \}$ is orthonormal, it follows that 
$$
g_{jk}|_{\gamma(t)} = \delta_{jk}.
$$
Thus also $\partial_1 g_{jk}|_{\gamma(t)} = 0$. We compute 
$$
\partial_{\alpha} g_{jk} = \partial_{\alpha} \langle \partial_j, \partial_k \rangle = \langle \nabla_{\partial_{\alpha}} \partial_j, \partial_k \rangle + \langle \partial_j, \nabla_{\partial_{\alpha}} \partial_k \rangle
$$
where $\nabla$ is the Levi-Civita connection. Since $\nabla_{\partial_j} \partial_k = \nabla_{\partial_k} \partial_j$, we have  
$$
\nabla_{\partial_{\alpha}} \partial_1|_{\gamma(t)} = \nabla_{\partial_1} \partial_{\alpha}|_{\gamma(t)} = \nabla_{\dot{\gamma}(t)} E_{\alpha}(t)|_{\gamma(t)} = 0
$$
because $E_{\alpha}$ is parallel along $\gamma$. Thus $\partial_{\alpha} g_{11}|_{\gamma(t)} = 0$, and also 
$$
\partial_{\alpha} g_{1 \beta}|_{\gamma(t)} = \langle \partial_1, \nabla_{\partial_{\alpha}} \partial_{\beta} \rangle|_{\gamma(t)} = \Gamma_{\alpha \beta}^1(\gamma(t))
$$
where $\Gamma_{jk}^l$ are the Christoffel symbols. Considering the geodesic $z(s) = \exp_{\gamma(t)}(sa^{\alpha} E_{\alpha}(t))$ for some constants $a^{\alpha}$, so that $z(s)$ is given in the $x$ coordinates by $z(s) = (t, sa^2, \ldots, sa^m)$, the geodesic equation 
$$
\ddot{z}^l(s) + \Gamma_{jk}^l(z(s)) \dot{z}^j(s) \dot{z}^k(s) = 0
$$
implies that for all $\alpha, \beta = 2, \ldots, m$ and $l = 1,\ldots,m$ we have 
$$
\Gamma_{\alpha \beta}^l(\gamma(t)) a^{\alpha} a^{\beta} = 0.
$$
Since $a^{\alpha}$ were arbitrary and $\Gamma_{\alpha \beta}^l = \Gamma_{\beta \alpha}^l$, we obtain 
$$
\Gamma_{\alpha \beta}^l(\gamma(t)) = 0.
$$
Thus $\partial_{\alpha} g_{1 \beta}|_{\gamma(t)} = 0$. Finally, 
\begin{align*}
\partial_{\alpha} g_{\beta \delta}|_{\gamma(t)} &= \langle \nabla_{\partial_{\alpha}} \partial_{\beta}, \partial_{\delta} \rangle + \langle \partial_{\beta}, \nabla_{\partial_{\alpha}} \partial_{\delta} \rangle|_{\gamma(t)} = \Gamma_{\alpha \beta}^{\delta}(\gamma(t)) + \Gamma_{\alpha \delta}^{\beta}(\gamma(t)) \\ 
 &= 0.
\end{align*}
We have proved that $\partial_i g^{jk}|_{\gamma(t)} = 0$ for all $i, j, k$.
\end{proof}

\begin{proof}[Proof of Proposition \ref{prop_gaussianbeam_quasimode}]
We begin by embedding $(M,g)$ in some closed manifold $(\widehat{M},g)$, and extend $\gamma$ as a unit speed geodesic in $\widehat{M}$. Let $\eps > 0$ be such that $\gamma(t) \in \widehat{M} \smallsetminus M$ for $t \in [-2\eps,0) \cup (L,L+2\eps]$ (here we use the fact that $\gamma$ is nontangential). Our purpose is to construct a Gaussian beam quasimode near $\gamma([-\eps,L+\eps])$.

Fix a point $p_0 = \gamma(t_0)$ on $\gamma([-\eps,L+\eps])$ and let $(t,y)$ be coordinates near $p_0$, defined in a set $U = \{ (t,y) \,;\, \abs{t-t_0} < \delta, \abs{y} < \delta' \}$, such that the geodesic near $p_0$ is given by $\Gamma = \{ (t,0) \,;\, \abs{t-t_0} < \delta \}$, and 
$$
g^{jk}|_{\Gamma} = \delta^{jk}, \quad \partial_i g^{jk}|_{\Gamma} = 0.
$$
Here we write $x = (t,y)$ where $t = x_1$ and $y = (x_2,\ldots,x_m)$. (Of course we will later use the coordinates in Lemma \ref{lemma_quasimode_coordinates}.) We will construct a quasimode $v_s$ concentrated near $\Gamma$, having the form 
$$
v_s = e^{is\Theta} a
$$
where $s = \tau+i\lambda$, and $\Theta$ and $a$ are smooth complex functions near $\Gamma$ with $a$ supported in $\{ \abs{y} < \delta'/2 \}$.

We compute 
$$
(-\Delta - s^2) v_s = f
$$
where 
$$
f = e^{is\Theta}(s^2 [(\langle d\Theta, d\Theta \rangle - 1)a] - is[2\langle d\Theta, da \rangle + (\Delta \Theta) a] - \Delta a).
$$
Here, the $g$-inner product $\langle \,\cdot\,,\,\cdot\, \rangle$ has been extended as a complex bilinear form to complex valued tensors. We first choose $\Theta$ so that 
\begin{equation} \label{theta_nth_order}
\langle d\Theta, d\Theta \rangle = 1 \quad \text{to $N$th order on $\Gamma$}.
\end{equation}
In fact we look for $\Theta$ of the form $\Theta = \sum_{j=0}^N \Theta_j$ where 
\begin{equation} \label{theta_taylor}
\Theta_j(t,y) = \sum_{\abs{\alpha} = j} \frac{\Theta_{j,\alpha}(t)}{\alpha!} y^{\alpha}.
\end{equation}
We also write $g^{jk} = \sum_{l=0}^N g^{jk}_l + r^{jk}_{N+1}$ where 
$$
g^{jk}_l(t,y) = \sum_{\abs{\beta} = l} \frac{g^{jk}_{l,\beta}(t)}{\beta!} y^{\beta}, \quad r^{jk}_{N+1} = O(\abs{y}^{N+1}).
$$
By the properties of our coordinates, $g^{jk}_0 = \delta^{jk}$ and $g^{jk}_1 = 0$.

Choose $\Theta_0(t) = t$ and $\Theta_1(t,y) = 0$. With the understanding that $j,k$ run from $1$ to $m$ and $\alpha,\beta$ run from $2$ to $m$, we have 
\begin{multline}
g^{jk} \partial_j \Theta \partial_k \Theta - 1 = (1+g^{11}_2+\ldots)(1+\partial_t \Theta_2+\ldots)(1+\partial_t \Theta_2+\ldots) \\
 + 2(g^{1\alpha}_2+\ldots)(1+\partial_t \Theta_2+\ldots)(\partial_{y_{\alpha}} \Theta_2+\ldots) \\
 + (\delta^{\alpha \beta}+g_2^{\alpha \beta}+\ldots)(\partial_{y_{\alpha}} \Theta_2+\partial_{y_{\alpha}} \Theta_3+\ldots)(\partial_{y_{\beta}} \Theta_2+\partial_{y_{\beta}} \Theta_3+\ldots) - 1 \\
  = [2\partial_t \Theta_2 + \nabla_y \Theta_2 \cdot \nabla_y \Theta_2 + g_2^{11}] \\
  + \sum_{p=3}^N \Big[ 2\partial_t \Theta_p + 2 \nabla_y \Theta_2 \cdot \nabla_y \Theta_p + \sum_{l=0}^p g_l^{11} \sum_{\overset{j+k=p-l}{j,k<p}} \partial_t \Theta_j \partial_t \Theta_k \\
  + 2 \sum_{l=2}^p g_l^{1 \alpha} \sum_{\overset{j+k=p+1-l}{k \geq 2}} \partial_t \Theta_j \partial_{\alpha} \Theta_k + \sum_{l=0}^{p-2} g_l^{\alpha \beta} \sum_{\overset{j+k=p+2-l}{2\leq j,k < p}} \partial_{\alpha} \Theta_j \partial_{\beta} \Theta_k \Big] \\
  + O(\abs{y}^{N+1}). \label{theta_equation_terms}
\end{multline}
In the last equality, we have grouped the terms in such a way that each quantity in brackets is a homogeneous polynomial in $y$ (the first term in brackets has degree $2$, and the others have degree $p$ for $p=3,\ldots,N$).

We will first choose $\Theta_2$ so that the first term in brackets vanishes. Writing $\Theta_2(t,y) = \frac{1}{2} H(t) y \cdot y$ where $H(t)$ is a smooth complex symmetric matrix, it follows that $H$ should satisfy the matrix Riccati equation 
$$
\dot{H}(t) + H(t)^2 = F(t)
$$
where $F(t)$ is the symmetric matrix such that $g^{11}_2(t,y) = -F(t)y \cdot y$. Choosing $H(t_0) = H_0$ where $H_0$ is some complex symmetric matrix with $\im(H_0)$ positive definite, it follows that the Riccati equation has a unique smooth complex symmetric solution $H(t)$ with $\im(H(t))$ positive definite \cite[Lemma 2.56]{KKL}. This completes the construction of $\Theta_2$.

We now look at the $p=3$ term in brackets in \eqref{theta_equation_terms}, and want to choose $\Theta_3$ so that this term becomes zero. The equation becomes 
$$
2 \partial_t \Theta_3 + 2 \nabla_y \Theta_2 \cdot \nabla_y \Theta_3 = F(t,y)
$$
where $F$ is a third order homogeneous polynomial in $y$ only depending on $\Theta_2$ and $g$. Writing $\Theta_3$ as in \eqref{theta_taylor}, this equation becomes a linear first order system of ODEs for the Taylor coefficients $\Theta_{3,\alpha}(t)$, and we can solve these equations uniquely by prescribing some initial conditions on $t_0$. Thus we have found $\Theta_3$, and repeating this argument we may find $\Theta_4,\ldots,\Theta_N$ successively by solving linear first order ODEs on $\Gamma$ with prescribed initial conditions at $t_0$. In this way, we obtain a smooth $\Theta$ satisfying \eqref{theta_nth_order}.

The next step is the find $a$ such that, up to a small error, 
$$
s[2\langle d\Theta, da \rangle + (\Delta \Theta) a] - i\Delta a = 0 \quad \text{to $N$th order on $\Gamma$}.
$$
We look for $a$ in the form 
$$
a = \tau^{\frac{m-1}{4}}(a_0 + s^{-1} a_{-1} + \ldots + s^{-N} a_{-N}) \chi(y/\delta')
$$
where $\chi$ is a smooth function with $\chi = 1$ for $\abs{y} \leq 1/4$ and $\chi = 0$ for $\abs{y} \geq 1/2$. Writing $\eta = \Delta \Theta$, it is sufficient to determine $a_j$ so that 
\begin{align*}
2\langle d\Theta,d a_0 \rangle + \eta a_0 &= 0 \quad \text{to $N$th order on $\Gamma$}, \\
2\langle d\Theta,d a_{-1} \rangle + \eta a_{-1} - i\Delta a_0 &= 0 \quad \text{to $N$th order on $\Gamma$}, \\
\vdots & \\
2\langle d\Theta,d a_{-N} \rangle + \eta a_{-N} - i\Delta a_{-(N-1)} &= 0 \quad \text{to $N$th order on $\Gamma$}.
\end{align*}
Consider $a_0 = a_{00} + \ldots + a_{0N}$ where $a_{0j}(t,y)$ is a polynomial of order $j$ in $y$, and similarly let $\eta = \eta_0 + \ldots + \eta_N$. The equation for $a_0$ becomes 
\begin{multline*}
2(1+g^{11}_2+\ldots)(1+\partial_t \Theta_2+\ldots)(\partial_t a_{00}+\partial_t a_{01}+\ldots) \\
 + 4(g^{1\alpha}_2+\ldots)(1+\partial_t \Theta_2+\ldots)(\partial_{y_{\alpha}} a_{01}+\partial_{y_{\alpha}} a_{02}+\ldots) \\
 + 2(\delta^{\alpha \beta}+g^{\alpha \beta}_2+\ldots)(\partial_{y_{\alpha}} \Theta_2+\partial_{y_{\alpha}}\Theta_3+\ldots)(\partial_{y_{\beta}} a_{01}+\partial_{y_{\beta}}a_{02}+\ldots) \\
 + (\eta_0+\eta_1+\ldots)(a_{00}+a_{01}+\ldots) \\
 = [2\partial_t a_{00}+\eta_0 a_{00}] + [2\partial_t a_{01}+2\nabla_y \Theta_2 \cdot \nabla_y a_{01} + \eta_0 a_{01} + \eta_1 a_{00}] + \ldots.
\end{multline*}
Here 
$$
\eta_0(t) = \Delta \Theta(t,0) = \partial_{y_{\alpha}}(H_{\alpha \beta}(t) y_{\beta}) = \tr\,H(t).
$$
We want to choose $a_{00}$ so that the first term in brackets vanishes, that is, 
$$
\partial_t a_{00} + \frac{1}{2} (\tr\,H(t)) a_{00} = 0.
$$
This has the solution 
$$
a_{00}(t) = c_0 e^{-\frac{1}{2} \int_{t_0}^t \tr\,H(s) \,ds}, \quad a_{00}(t_0) = c_0.
$$
For later purposes we choose the constant as 
\begin{equation} \label{gaussianbeam_c0_choice}
c_0 = \frac{\sqrt[4]{\det \im(H(t_0))}}{\sqrt{\int_{\mR^{m-1}} e^{-\abs{y}^2} \,dy}}.
\end{equation}
We obtain $a_{01},\ldots,a_{0N}$ successively by solving linear first order ODEs with prescribed initial conditions at $t_0$. The functions $a_1,\ldots,a_N$ may be determined in a similar way so that the required equations are satisfied to $N$th order on $\Gamma$. This completes the construction of $a$.

We have constructed a function $v_s = e^{is\Theta} a$ in $U$ where 
\begin{align*}
\Theta(t,y) &= t + \frac{1}{2} H(t)y \cdot y + \widetilde{\Theta}, \\
a(t,y) &= \tau^{\frac{m-1}{4}}(a_0 + s^{-1} a_{-1} + \ldots + s^{-N} a_{-N}) \chi(y/\delta'), \\
a_0(t,0) &= c_0 e^{-\frac{1}{2} \int_{t_0}^t \tr\,H(s) \,ds}.
\end{align*}
Here $\widetilde{\Theta} = O(\abs{y}^3)$ and $\Theta$ and each $a_j$ are independent of $\tau$. Also, $f = (-\Delta-s^2)v_s$ is of the form 
\begin{multline*}
f = e^{is\Theta} \tau^{\frac{m-1}{4}}(s^2 h_2 a + s h_1 + \ldots + s^{-(N-1)} h_{-(N-1)} - s^{-N} \Delta a_{-N}) \chi(y/\delta') \\
+ e^{is\Theta} \tau^{\frac{m-1}{4}} s b \widetilde{\chi}(y/\delta')
\end{multline*}
where for each $j$ one has $h_j = 0$ to $N$th order on $\Gamma$, $b$ vanishes near $\Gamma$, and $\widetilde{\chi}$ is a smooth function with $\widetilde{\chi} = 0$ for $\abs{y} \geq 1/2$.

To prove the norm estimates for $v_s$ in $U$, note that 
\begin{align*}
\abs{e^{is\Theta}} &= e^{-\lambda \re\,\Theta} e^{-\tau \im\,\Theta} = e^{-\lambda t} e^{-\frac{1}{2}\tau \im(H(t))y \cdot y} e^{-\lambda O(\abs{y}^2)} e^{-\tau O(\abs{y}^3)}.
\end{align*}
Here $\im(H(t))y \cdot y \geq c\abs{y}^2$ for $(t,y) \in U$ where $c>0$ depends on $H_0$ and $\delta$. This implies that for $t$ in a compact interval, after decreasing $\delta'$ if necessary, we have 
$$
\abs{v_s(t,y)} \lesssim \tau^{\frac{m-1}{4}} e^{-\frac{1}{4}c \tau \abs{y}^2} \chi(y/\delta').
$$
This shows that 
\begin{gather*}
\norm{v_s}_{L^2(U)} \lesssim \norm{\tau^{\frac{m-1}{4}} e^{-\frac{1}{4}c \tau \abs{y}^2}}_{L^2(U)} = O(1), \\
\norm{(-\Delta-s^2)v_s}_{L^2(U)} \lesssim \norm{\tau^{\frac{m-1}{4}} e^{-\frac{1}{4}c \tau \abs{y}^2}(\tau^2 \abs{y}^{N+1} + \tau^{-N})}_{L^2(U)} = O(\tau^{\frac{3-N}{2}})
\end{gather*}
as $\tau \to \infty$. The norm estimates for $v_s$ in $U$ follow upon replacing $N$ by $2K+3$.

For later purposes we record an additional estimate: if $U \cap \partial M \neq \emptyset$, the fact that the geodesic is nontangential allows to write $\partial M$ locally in the $(t,y)$ coordinates as $\{ (t(y),y) ; \abs{y} < \varepsilon \}$ for some smooth function $t = t(y)$. By choosing $\delta'$ small enough, we then have 
\begin{align}
\norm{v_s}_{L^2(\partial M)}^2 &= \int_{\abs{y} < \varepsilon} \abs{v_s(t(y),y)}^2 \,dS(y) \lesssim \int_{\mR^{m-1}} \tau^{\frac{m-1}{2}} e^{-\frac{1}{2}c \tau \abs{y}^2} \,dy \notag \\
 &= O(1) \label{vs_boundary_estimate}
\end{align}
as $\tau \to \infty$.

We will now construct the quasimode $v_s$ in $M$ by gluing together quasimodes defined on small pieces. Let $\gamma([-\eps,L+\eps])$ be covered by open sets $U^{(0)}, \ldots, U^{(r)}$ as in Lemma \ref{lemma_quasimode_coordinates} corresponding to intervals $I^{(j)}$ (with the same $\delta'$ for each $U^{(j)}$) such that one can find quasimodes in each $U^{(j)}$. We first find a function $v_s^{(0)} = e^{is\Theta^{(0)}} a^{(0)}$ in $U^{(0)}$ as above, with some fixed initial conditions at $t=-\eps$ for the ODEs determining $\Theta^{(0)}$ and $a^{(0)}$. Choose some $t_0'$ with $\gamma(t_0') \in U^{(0)} \cap U^{(1)}$, and construct a quasimode $v_s^{(1)} = e^{is\Theta^{(1)}} a^{(1)}$ in $U^{(1)}$ by choosing the initial conditions for the ODEs for $\Theta^{(1)}$ and $a^{(1)}$ at $t_0'$ to be the corresponding values of $\Theta^{(0)}$ and $a^{(0)}$ at $t_0'$. Continuing in this way we obtain $v_s^{(2)}, \ldots, v_s^{(r)}$. Let $\{\chi_j(t)\}$ be a partition of unity near $[-\eps,L+\eps]$ corresponding to the intervals $\{I^{(j)}\}$, let $\widetilde{\chi}_j(t,y) = \chi_j(t)$ in $U^{(j)}$, and define 
$$
v_s = \sum_{j=0}^r \widetilde{\chi}_j v_s^{(j)}.
$$
Note that the ODEs for the phase functions and amplitudes have the same initial data in $U^{(j)}$ and in $U^{(j+1)}$, which shows that we actually have $v_s^{(j)} = v_s^{(j+1)}$ in $U^{(j)} \cap U^{(j+1)}$. In particular, if $p_1, \ldots, p_R$ are the distinct points where the geodesic self-intersects, if $0 \leq t_1 < \ldots < t_{R'}$ are the times when the geodesic self-intersects, and if $V_1,\ldots,V_R$ are small balls centered at $p_j$, then choosing $\delta'$ small enough we have a covering 
$$
\supp(v_s) \cap M \subset \left( \cup_{j=1}^R V_j \right) \cup \left( \cup_{k=1}^{S} W_k \right)
$$
where, in each $V_j$, the quasimode is a finite sum 
$$
v_s|_{V_j} = \sum_{\gamma(t_l) = p_j} v_s^{(l)},
$$
and in each $W_k$ there is some $l(k)$ so that the quasimode is given by 
$$
v_s|_{W_k} = v_s^{l(k)}.
$$
This shows that $L^2$ bounds for $v_s$ and $(-\Delta-s^2) v_s$ in $M$ follow from the corresponding bounds for each $v_s^{(l)}$.

We still need to verify the limit 
$$
\int_{M} \abs{v_{\tau+i\lambda}}^2 \psi \,dV_{g} \to \int_0^L e^{-2\lambda t} \psi(\gamma(t)) \,dt \quad \text{as $\tau \to \infty$}
$$
for any $\psi \in C(M)$. By a partition of unity, it is enough to consider functions $\psi$ with $\psi \in C_c(V_j \cap M)$ and $\psi \in C_c(W_k \cap M)$ (thus $\psi$ may be nonzero on $\partial M$). Let us begin with the case where $\psi \in C_c(W_k \cap M)$ for some $k$. Then $v_s = e^{is\Theta} a$ where $\Theta = t + \frac{1}{2}H(t)y \cdot y + O(\abs{y}^3)$ and $a = \tau^{\frac{m-1}{4}}(a_0 + O(\tau^{-1})) \chi(y/\delta')$. Let $\rho = \abs{g}^{1/2}$. We have 
\begin{align*}
 &\int_{M} \abs{v_{\tau+i\lambda}}^2 \psi \,dV_{g} \\
 &= \int_0^L \int_{\mR^{m-1}} e^{-2\lambda t} e^{-\tau \im(H(t))y \cdot y} e^{\tau O(\abs{y}^3)} e^{O(\abs{y}^2)} \tau^{\frac{m-1}{2}} (\abs{a_0}^2 + O(\tau^{-1})) \chi(y/\delta')^2 \psi \rho \,dt \,dy \\
 &= \int_0^L e^{-2\lambda t} \int_{\mR^{m-1}} e^{-\im(H(t))y \cdot y} e^{\tau^{-1/2} O(\abs{y}^3)} e^{\tau^{-1} O(\abs{y}^2)} \times \\
 &\qquad (\abs{a_0(t,\tau^{-1/2}y)}^2 + O(\tau^{-1})) \chi(y/\tau^{1/2} \delta')^2 \psi(t,\tau^{-1/2} y) \rho(t,\tau^{-1/2}y) \,dt \,dy.
\end{align*}
Since $\im(H(t))$ is positive definite and $\delta'$ is sufficiently small, the term $e^{-\im(H(t))y \cdot y}$ dominates the other exponentials and one obtains 
\begin{align*}
 &\lim_{\tau \to \infty} \int_{M} \abs{v_{\tau+i\lambda}}^2 \psi \,dV_{g} \\
 &= \int_0^L e^{-2\lambda t} \left( \int_{\mR^{m-1}} e^{-\im(H(t))y \cdot y} \,dy \right) \abs{a_0(t,0)}^2 \psi(t,0) \rho(t,0) \,dt.
\end{align*}
Evaluating the integral over $y$ and using that $\rho(t,0) = 1$ gives 
\begin{multline*}
\lim_{\tau \to \infty} \int_{M} \abs{v_{\tau+i\lambda}}^2 \psi \,dV_{g} \\
 = \left( \int_{\mR^{m-1}} e^{-\abs{y}^2} \,dy \right) \int_0^L e^{-2\lambda t} \frac{\abs{a_0(t,0)}^2}{\sqrt{\det\,\im(H(t))}} \psi(t,0) \,dt.
\end{multline*}
Here $a_0(t,0) = a_0(t_0,0) e^{-\frac{1}{2} \int_{t_0}^t \tr\,H(s) \,ds}$. Now we use the fact in \cite[Lemma 2.58]{KKL} that solutions of the matrix Riccati equation have the property 
$$
\det \im(H(t)) = \det \im(H(t_0)) e^{-2 \int_{t_0}^t \tr\,\re(H(s)) \,ds}.
$$
It follows that $\abs{a_0(t,0)}^2/\sqrt{\det\,\im(H(t))}$ is constant in time. The choice \eqref{gaussianbeam_c0_choice} fixes this constant and proves the limit for $\psi \in C_c(W_k \cap M)$.

Now assume that $\psi \in C_c(V_j \cap M)$, so that 
$$
v_s = \sum_{\gamma(t_l) = p_j} v_s^{(l)} \quad \text{in } \supp(\psi), \quad v_s^{(l)} = e^{is\Theta^{(l)}} a^{(l)}.
$$
It follows that 
$$
\abs{v_s}^2 = \sum_{\gamma(t_l) = p_j} \abs{v_s^{(l)}}^2 + \sum_{\underset{\gamma(t_l) = \gamma(t_l') = p_j}{l \neq l'}} v_s^{(l)} \overline{v_s^{(l')}}.
$$
The computation above gives the right limit for each $\abs{v_s^{(l)}}^2$ term. Therefore, it is enough to show that limits for the cross terms vanish as $\tau \to \infty$.

Since all self-intersections must be transversal, and since $d\Theta^{(l)}(\gamma(t_l))$ is the covector corresponding to $\dot{\gamma}(t_l)$ with respect to the metric, we may assume (by decreasing the sets $V_j$ in the original construction if necessary) that $\re(d\Theta^{(l)} - d\Theta^{(l')})$ is nonvanishing in $V_j$ if $\gamma(t_l) = \gamma(t_{l'}) = p_j$ but $l \neq l'$. The cross terms lead to terms of the form 
$$
\int_{V_j \cap M} v^{(l)} \overline{v^{(l')}} \psi \,dV = \int_{V_j \cap M} e^{i\tau \phi} w^{(l)} \overline{w^{(l')}} \psi \,dV
$$
where $\phi = \re(\Theta^{(l)} - \Theta^{(l')})$ has nonvanishing gradient in $V_j$, and $w^{(r)} = e^{is \im(\Theta^{(r)})} e^{-\lambda \re(\Phi^{(r)})} a^{(r)}$. We wish to prove that 
\begin{equation} \label{vr_crossterm_limit}
\lim_{\tau \to \infty} \int_{V_j \cap M} e^{i\tau \phi} w^{(l)} \overline{w^{(l')}} \psi \,dV = 0, \quad l \neq l',
\end{equation}
showing that the cross terms vanish in the limit. To show \eqref{vr_crossterm_limit}, let $\eps > 0$, and decompose $\psi = \psi_1 + \psi_2$ where $\psi_1 \in C^{\infty}_c(V_j \cap M)$ ($\psi_1$ may be nonzero on $\partial M$) and $\norm{\psi_2}_{L^{\infty}(V_j \cap M)} \leq \eps$. Then 
$$
\abs{\int_{V_j \cap M} e^{i\tau \phi} w^{(l)} \overline{w^{(l')}} \psi_2 \,dV} \lesssim \norm{w^{(l)}}_{L^2} \norm{w^{(l')}}_{L^2} \norm{\psi_2}_{L^{\infty}} \lesssim \eps
$$
since $\norm{w^{(r)}}_{L^2} \lesssim \norm{v^{(r)}}_{L^2} \lesssim 1$. For the smooth part $\psi_1$, we employ a non-stationary phase argument and integrate by parts using that 
$$
e^{i\tau \phi} = \frac{1}{i\tau} L(e^{i\tau \phi}), \quad Lw = \langle \abs{d\phi}^{-2} d\phi, dw \rangle.
$$
This gives 
\begin{multline*}
\int_{V_j \cap M} e^{i\tau \phi} w^{(l)} \overline{w^{(l')}} \psi_1 \,dV = \int_{\partial M} \frac{\partial_{\nu} \phi}{i\tau\abs{d\phi}^2} e^{i\tau \phi} v^{(l)} \overline{v^{(l')}} \psi_1 \,dS \\
+ \frac{1}{i\tau} \int_{V_j \cap M} e^{i\tau \phi} L^t(w^{(l)} \overline{w^{(l')}} \psi_1) \,dV.
\end{multline*}
Since $\norm{v^{(r)}}_{L^2(\partial M)} = O(1)$ by \eqref{vs_boundary_estimate}, the boundary term can be made arbitrarily small as $\tau \to \infty$. As for the last term, the worst behavior is when the transpose $L^t$ acts on $e^{is\im(\Theta^{(r)})}$, and these terms have bounds of the form 
$$
\norm{\abs{d(\im(\Theta^{(l)}))} v^{(l)}}_{L^2} \norm{v^{(l')}}_{L^2} \norm{\psi_1}_{L^{\infty}}.
$$
Here $\abs{d(\im(\Theta^{(l)}))} \lesssim \abs{y}$ if $(t,y)$ are coordinates along the geodesic segment corresponding to $v^{(l)}$, and the computation above for $\norm{v^{(l)}}_{L^2}$ shows that 
$$
\norm{\abs{d(\im(\Theta^{(l)}))} v^{(l)}}_{L^2} \norm{v^{(l')}}_{L^2} \norm{\psi_1}_{L^{\infty}} \lesssim \tau^{-1/2}.
$$
This finishes the proof of \eqref{vr_crossterm_limit}.
\end{proof}

In the end of this section, we switch back to writing $(M_0,g_0)$ for the transversal manifold. Instead of using injectivity for the attenuated ray transform (see \cite[Section 7]{DKSaU} and \cite{SaU} for injectivity results), we will reduce matters to the unattenuated ray transform.

\begin{proof}[Proof of Theorem \ref{claim4}]
This is exactly Proposition \ref{prop_gaussianbeam_quasimode}.
\end{proof}

\begin{proof}[Proof of Theorem \ref{claim2}]
Assume the conditions in Theorem \ref{claim2}, and write $q = c(q_1-q_2)$. As discussed in the end of the proof of Theorem \ref{claim3}, we can extend $q$ by zero to $\mR \times M_0$ so that the extension, also denoted by $q$, is in $C_c(\mR \times M_0)$. Now, the combination of Theorems \ref{claim3} and \ref{claim4} implies that 
\begin{equation} \label{attenuated_raytransform_vanishing}
\int_{\gamma} \widehat{q}(2\lambda,\gamma(t)) e^{-2\lambda t} \,dt = 0
\end{equation}
for any $\lambda \in \mR$ and for any nontangential geodesic $\gamma$ in $M_0$, where  
$$
\widehat{q}(2\lambda,x') = \int_{-\infty}^{\infty} e^{-2i\lambda x_1} q(x_1,x') \,dx_1.
$$
Thus the attenuated geodesic ray transform of $\widehat{q}(2\lambda,\,\cdot\,)$, with constant attenuation $-2\lambda$, vanishes over all nontangential geodesics in $M_0$.

Assume now that the unattenuated ray transform in $M_0$ (the case $\lambda=0$) is injective. Evaluating \eqref{attenuated_raytransform_vanishing} at $\lambda=0$ shows that 
$$
\int_{\gamma} \widehat{q}(0,\gamma(t)) \,dt = 0
$$
for all nontangential geodesics $\gamma$. Injectivity of the ray transform then gives that $\widehat{q}(0,\,\cdot\,) = 0$ in $M_0$. Next we differentiate \eqref{attenuated_raytransform_vanishing} with respect to $\lambda$ and evaluate at $\lambda=0$, to obtain 
$$
\int_{\gamma} \left[ 2 \frac{\partial \widehat{q}}{\partial \lambda}(0,\gamma(t)) - 2 t \widehat{q}(0,\gamma(t)) \right] \,dt = 0.
$$
But since $\widehat{q}(0,\,\cdot\,) = 0$, this implies the vanishing of the ray transform of $\frac{\partial \widehat{q}}{\partial \lambda}(0,\,\cdot\,)$ and hence also the vanishing of $\frac{\partial \widehat{q}}{\partial \lambda}(0,\,\cdot\,)$ in $M_0$. Taking higher derivatives with respect to $\lambda$ in \eqref{attenuated_raytransform_vanishing} and continuing this argument implies that 
$$
\left( \frac{\partial}{\partial \lambda} \right)^k  \widehat{q}(0,x') = 0 \quad \text{for all } x' \in M_0
$$
and for all $k \geq 0$. Using that $\widehat{q}(\,\cdot\,,x')$ is analytic as the Fourier transform of a compactly supported function, we see that $\widehat{q}(\xi_1,x') = 0$ for all $\xi_1 \in \mR$ and $x' \in M_0$. Thus $q = 0$, or $q_1 = q_2$ as required.
\end{proof}

\begin{proof}[Proof of Theorem \ref{claim1}]
First note that $C_g = \{ (f, \Lambda_g f) \,;\, f \in H^{1/2}(\partial M) \}$, where $\Lambda_g$ is the DN map 
$$
\Lambda_g: u|_{\partial M} \mapsto \partial_{\nu} u|_{\partial M}, \quad \Delta_g u = 0 \text{ in } M.
$$
If $(M,g_1)$ and $(M,g_2)$ are two CTA manifolds in the same conformal class with $C_{g_1} = C_{g_2}$, we write $g_2 = g$ and $g_1 = cg$ where $c$ is some positive function. Then 
$$
\Lambda_{cg} = \Lambda_g.
$$
Boundary determination \cite[Proposition 8.1]{DKSaU} implies that $c|_{\partial M} = 1$ and $\partial_{\nu} c|_{\partial M} = 0$. Using \cite[Proposition 8.2]{DKSaU}, this implies the following identity for DN maps of Schr\"odinger equations in $(M,g)$, 
$$
\Lambda_{g,-cq_c} = \Lambda_{g,0},
$$
where $q_c = c^{\frac{n-2}{4}} \Delta_{cg}(c^{-\frac{n-2}{4}})$. Since $(M,g)$ is a CTA manifold and the ray transform in the transversal manifold was assumed to be injective, we can now use Theorem \ref{claim2} to conclude uniqueness of the potentials, $-cq_c = 0$. But this implies that 
$$
\Delta_{cg}(c^{-\frac{n-2}{4}}) = 0 \text{ in } M, \quad c^{-\frac{n-2}{4}}|_{\partial M} = 1.
$$
Uniqueness in the Dirichlet problem implies that $c = 1$ in $M$, which shows that $g_1 = g_2$.
\end{proof}

In the next section, we outline an alternative method for constructing quasimodes concentrating near a geodesic. The method is based on microlocal reductions instead of the direct construction that was given above. 

\section{Microlocal construction} \label{sec_microlocal}

Another possible approach to constructing quasimodes is a microlocal one; canonical quantization by a Fourier integral operator allows one to reduce the semiclassical operator $\Delta_{g}+s^2$ to a simple form and construct the corresponding quasimodes. It will be convenient to use semiclassical conventions, and choose $h=\tau^{-1}$ as a small parameter.
We hope that this general construction might help to better understand the problem by providing a different viewpoint, and might suggest a way to answer Question \ref{question_intro} in the introduction perhaps via some other normal forms.
We refer to \cite{Mart} and \cite{Zworski} for a general presentation in semiclassical analysis. Let us nevertheless begin, for the convenience of the reader and to set our notations, by recalling a few definitions and results, which we will need in our exposition. 

\subsection{Elements of semiclassical analysis}

Semiclassical Sobolev spaces $H^k_{\rm scl}$ on a  closed Riemannian manifold 
(or in Euclidean space) are defined like classical Sobolev spaces but are endowed with the following norms depending on the semiclassical parameter $h \in (0,1]$,
    $$ \|u\|_{H^k_{\rm scl}} = \bigg(\sum_{j=0}^k  \|( h \nabla)^j u\|^2_{L^2}\bigg)^{\frac{1}{2}} $$
where $\nabla^j$ are covariant derivatives on the Riemannian manifold. Semiclassical symbols of order $k$ on $T^*\R^m$ are smooth functions $a$ on $\R^{2m}$ depending on a parameter $h \in (0,1]$ for which for all multiindices $(\alpha,\beta) \in \mN \times \mN$ we have 
     $$ C_{\alpha \beta} = \sup_{(x,\xi) \in T^*\R^m, h \in (0,1]} (1+|\xi|^2)^{\frac{-k+|\beta|}{2}} |\partial_x^{\alpha}\partial_{\xi}^{\beta} a(x,\xi,h)| < \infty. $$
The linear space of such symbols is denoted $S^k_{\rm scl}(T^*\R^m)$. Pseudodifferential operators are defined through the semiclassical Weyl quantization
\begin{align*}
     (\op{h}a) u (x) = (2\pi h)^{-n} \iint e^{\frac{i}{h}(x-y) \cdot \xi} a\bigg(\frac{x+y}{2},\xi,h\bigg) u(y) \, dy \, d\xi
\end{align*}
of a symbol $a \in S^k_{\rm scl}(T^*\R^m)$ and we denote $\Psi_{\rm scl}^k(\R^m)$ the corresponding space of operators. Symbols on the cotangent bundle on a compact manifold are smooth functions on $T^*\widehat{M} \times (0,1]$  which after cutoff to a coordinate patch pull back under local coordinates to symbols on $T^*\R^m$. Pseudodifferential operators of order $k$ on a compact manifold $\widehat{M}$ are operators $A_h:C^{\infty}(\widehat{M}) \to C^{\infty}(\widehat{M})$ such that 
for all pairs of coordinate patches $U,V$ and all cutoff functions $\varphi \in C^{\infty}_c(U),\psi \in C^{\infty}_c(V)$
\begin{itemize}
      \item $\|\psi A_h \varphi\|_{L(H^{-N}_{\rm scl},H^N_{\rm scl})} = O(h^{\infty})$ for all integers $N$ if the supports of $\varphi$ and $\psi$ are disjoint,
      \item $\psi A_h \varphi$ written in local coordinates is a pseudodifferential operator $\op{h} a$ on $\R^m$ with symbol $a \in S^k(T^*\R^m)$.
\end{itemize}
We write $\Psi^k_{\rm scl}(\widehat{M})$ for the linear space of semiclassical pseudodifferential operators of order $k$ on $\widehat{M}$. Using a partition of unity and local coordinates, it is possible to quantize any semiclassical symbol $a \in S^k_{\rm scl}(T^*\widehat{M})$ into a pseudodifferential operator $\op{h} a \in \Psi^k_{\rm scl}(\widehat{M})$. Conversely, one can define a map which to any pseudodifferential operator $A_h \in \Psi^k_{\rm scl}(\widehat{M})$ associates a class $[a]$ of symbols in $S^k_{\rm scl}(\widehat{M})/hS^{k-1}_{\rm scl}(\widehat{M})$ called the semiclassical principal symbol of $A$ such that $A_h-\op{h} a \in h\Psi_{\rm scl}^{k-1}(\widehat{M})$. As usual one identifies a class of symbols with any of its representatives.
\begin{definition}
      A family $u =\{u_h\}_{0<h\leq h_0}$ of distributions on a closed compact manifold 
      $\widehat{M}$ or on $\R^m$ is said to be tempered  if there exists an integer $N$ such that $\|u_h\|_{H^{-N}_{\rm scl}}=O(h^{-N}$).

      The semiclassical wavefront set $\WF(u)$ of a tempered family of distributions $u =\{u_h\}_{0<h\leq h_0}$ on a compact manifold $\widehat{M}$ (resp. $\R^m$) is the complement of the set of points $(x_0,\xi_0) \in T^*\widehat{M}$ (resp. $T^*\R^m$) for which there exists a symbol $a \in S^0_{\rm scl}$ such that, for some constant $c > 0$ independent of $h$, one has $|a(x_0,\xi_0)| \geq c$ and 
        $$ \|(\op{h} a) u_h\|_{L^2} = O(h^{\infty}). $$
\end{definition}
If $A_h=\op{h} a$, one traditionally denotes by $\WF(A_h)$ the essential support of $a$, i.e. the complement of points $(x,\xi)$ in the cotangent bundle for which $\partial_x^{\alpha}\partial_{\xi}^{\beta}a=O(h^{\infty})$ near $(x,\xi)$ for all $\alpha,\beta$.

In the Euclidean space $\R^m$, there is an equivalent definition involving the semiclassical Fourier transform
       $$ \mathcal{F}_hu(\xi) = \int e^{-\frac{i}{h}y \cdot \xi} u(y) \, dy. $$
 \begin{definition}
      A point $(x_0,\xi_0) \in T^*\R^m$  does not belong to the semiclassical wavefront set of a tempered family of distributions on $\R^m$  if there exist smooth compactly supported functions $\chi,\psi$ which equal 1 near $x_0$ and $\xi_0$, respectively, such that 
          $$ \psi \mathcal{F}_h(\chi u) = O(h^{\infty}). $$
 \end{definition}
 
 \begin{remark}
      From the previous definition, the behaviour of the semiclassical wavefront set with respect to tensor products is clear:
          $$ \WF( u \otimes w) = \{(x,y,\xi,\eta) : (x,\xi) \in \WF(u), \, (y,\eta) \in \WF(v) \}. $$
 \end{remark}
 
 \begin{example}
      The following examples of semiclassical wavefront sets of functions in the Euclidean space will be useful to our purposes: both are easily deducible from the definition involving the semiclassical Fourier transform. 
      \begin{enumerate}
           \item[1.] Coherent states or wave packets 
               $$ \WF\big( (\pi h)^{-\frac{m}{4}} e^{-\frac{1}{2h}|x-x_0|^2+\frac{i}{h}(x-x_0)\cdot \xi_0}\big) = \{(x_0,\xi_0)\} $$
               (This is example (i) page 195 section 8.4.2 in \cite{Zworski}).
           \item[2.] Smooth functions independent of the semiclassical parameter $h$ : 
               $$ \WF(u) = \supp \, u \times \{0\} $$ 
               (see Remark (ii) page 195 section 8.4.2 in \cite{Zworski}). 
      \end{enumerate}
 \end{example} 
 
We also recall the action of semiclassical Fourier integral operators whose canonical relation is the graph of a canonical transformation. Fourier integral operators are operators whose Schwartz kernels are semiclassical Lagrangian distributions associated 
with a Lagrangian manifold. We will consider Fourier integral operators associated with a Lagrangian manifold which is the graph of a canonical transformation.
We denote $\pi_1 : T^*\R^m \to \R^m$ the first projection. 
A Fourier integral operator of order $k$ associated with the graph 
      $$ G = \big\{(x,\xi,\varsigma(x,\xi)) : (x,\xi) \in V\big\} $$ 
of a canonical transformation $\varsigma : V \to W$ between two open sets $V,W$ of $T^*\R^m$ is an operator which maps distributions on $X=\pi_1(V)$ to distributions on $Y=\pi_1(W)$ whose kernel can be written modulo a smooth function which is $O(h^{\infty})$ as the sum of terms of the form
\begin{align*}
    U_h(x,y) = (2\pi h)^{-m} \int e^{\frac{i}{h}(\varphi(x,\xi)-y \cdot \xi)} a(x,\xi,h) \, d\xi
\end{align*}
where $a \in S^k(T^*\R^m)$ and $\varphi$ is a generating function of the canonical transformation $\varsigma$. We recall that a function $\varphi : \widetilde{V} \to \widetilde{W}$ is a generating function of $G$ if 
     $$ G =\big\{(x,\partial_x\varphi(x,\xi),\partial_{\xi}\varphi(x,\xi),\xi) : (x,\xi) \in \widetilde{V} \big\} $$
     in other words, the relation between the canonical transformation $\varsigma$ and the generating function $\varphi$ is given by 
     $$ \varsigma(x,\partial_x\varphi) = (\partial_{\xi} \varphi,\xi). $$
For notational purposes, one needs to introduce the twisted relation
     $$ G' = \big\{(x,y,\xi,-\eta) : (x,\xi,y,\eta) \in G \big\}. $$
Indeed, the semiclassical wavefront set of the kernel $U_h$ is contained in $G' \subset T^*(X \times Y)$.

One denotes $I^k_{\rm scl}(X\times Y,G')$ the space of such Fourier integral operators. 
The reason for adopting this notation is that one abuses notations by identifying Fourier integral operators with their kernels which are semiclassical Lagrangian distributions on $X \times Y$ with semiclassical wave front set contained in the Lagrangian submanifold $G'$ of $T^*(X \times Y)$.
If one or the other of the sets $X,Y$ is a manifold without boundary, then the previous form has to be understood in local coordinates in $x$ or $y$.

\begin{lemma}
\label{PropagationLemma}
      Let $U_h \in I^k_{\rm scl}(X \times Y ,G')$ be a Fourier integral operator associated with the graph $G$ of a canonical transformation $\varsigma : V \mapsto W$ (mapping distributions on an open set $X=\pi_1(V) \subset \R^m$ to an open set $Y = \pi_1(W) \subset \widehat{M}$), then the semiclassical wavefront set is transformed under the action of $U_h$ in the following way
      $$ \WF(U_h u_h) \subset \varsigma(\WF(u_h) \cap V). $$ 
\end{lemma}
We will also need a semiclassical version of Egorov's theorem.
\begin{thm}
     Let $U_h \in I^k_{\rm scl}(\R^m \times \widehat{M} ,G')$, $V_h \in I^{-k}_{\rm scl}(\widehat{M} \times \R^m,(G^{-1})')$ be two semiclassical Fourier integral operators respectively associated with the graph $G$ of the canonical transformation $\varsigma$ and the graph $G^{-1}$ of $\varsigma^{-1}$, and $A \in \Psi^l_{\rm scl}(\widehat{M})$ a pseudodifferential operator  then $V_hAU_h$ is a pseudodifferential operator in $\widehat{M}$ with principal symbol $\chi (\varsigma^*a)$ where $\chi$ is the principal symbol of the pseudodifferential operator $V_hU_h \in \Psi^0_{\rm scl}(\widehat{M})$.
\end{thm}
In the classical setting, this is Theorem 25.3.5 in \cite{H3}, for semiclassical versions one can refer to \cite[Theorem 4.7.8]{Ler} and \cite[Theorem 11.5]{Zworski}.

\subsection{Semiclassical defect measures} 

It is time to introduce the notion of semiclassical defect measures which our introduction evoked and which lifts the measure used in our proofs to the cotangent bundle. Let $(M,g)$ be a compact Riemannian manifold with boundary. We refer to \cite{Bu} for a survey on semiclassical measures and to \cite[Section 5.3]{Zworski}.

\begin{definition}
     Let $(v_j)_{j=1}^{\infty}$ be a bounded sequence of $L^2$ functions on $M$ and $(h_j)_{j=1}^{\infty}$ a sequence of reals in $(0,1]$ (called a sequence of scales) converging to $0$. There exist subsequences $(v_{j_k})_{k=1}^{\infty}$, $(h_{j_k})_{k=1}^{\infty}$ and a positive Radon measure $\mu$ on $T^*M^{\rm int}$  such that for all $a \in C^{\infty}_c(T^*M^{\rm int})$
     \begin{align*}
          \lim_{k \to \infty} \int_M A_{h_{j_k}}v_{j_k} \, \overline{v_{j_k}} \, d V= \int_{T^*M} a \, d\mu
     \end{align*} 
     where $A_{h_{j_k}}$ is a semiclassical pseudodifferential operator with principal symbol $a$ and parameter $h_{j_k}$. Such a measure is called a semiclassical defect measure associated to the sequences $(v_j)_{j=1}^{\infty}$ and $(h_j)_{j=1}^{\infty}$.
\end{definition}
We are interested in the semiclassical defect measures associated with our family of quasimodes $(v_{s_j})_{j=1}^{\infty}$ for a sequence $s_j = h_j^{-1}+i\lambda$
with $(h_j)_{j=1}^{\infty}$ converging to $0$
\begin{align}
\label{QuasimodeSeq}
\norm{(-h^2_j\Delta_{g} - (1+ i \lambda h_j)^2) v_j}_{L^2(M)} = o(h_j), \quad \norm{v_j}_{L^2(M)} = O(1)
\end{align}
as $j \to \infty$.
\begin{lemma}
      All semiclassical measures associated to the sequence of quasimodes \eqref{QuasimodeSeq} are supported in the cosphere bundle $S^*M^{\rm int}$.
\end{lemma}
This is a consequence of Theorem 5.3 in \cite{Zworski} since the semiclassical principal symbol of $-h^2\Delta_g-(1+i\lambda h)^2$ is $|\xi|^2_g-1$. The adaptation to the manifold case is straightforward.
\begin{lemma}
     All semiclassical measures associated to the sequence of quasimodes \eqref{QuasimodeSeq} satisfy the following transport equation
     \begin{align*}
           {}^t(H_{p})\mu = 4 \lambda \mu
     \end{align*}
     where $H_p$ is the Hamiltonian vector field of the symbol $p(x,\xi)=|\xi|^2_g=g^{jk}(x) \xi_j\xi_k$.
\end{lemma}
\begin{proof}
     Let $a \in C^{\infty}_c(T^*M^{\rm int})$ be real-valued, and choose $\psi \in C^{\infty}_c(M^{\rm int})$ which equals one on the projection of $\supp \, a$ on $M^{\rm int}$. Since $a$ is real valued, the pseudodifferential operator $A_{h_j}=\op{h_j}a$ is self-adjoint and we have 
     \begin{align}
     \label{CommutEq}
          \frac{1}{ih_j} \big( [A_{h_j},-h_j^2\Delta_g] \psi v_j, \psi v_j \big)_{L^2(M)} &= \frac{2}{h_j} \im \big( -h_j^2\Delta_g (\psi v_j),A_{h_j}(\psi v_j) \big)_{L^2(M)} \nonumber \\ &= \frac{2}{h_j} \im \big( -h_j^2\Delta_g v_j,A_{h_j} v_j \big)_{L^2(M)} \nonumber + O(h_j^{\infty}) \nonumber \\ &= 4 \lambda \big( A_{h_j}v_j,v_j \big)_{L^2(M)} + o(1).
     \end{align}
     The principal symbol of the commutator $ih_j^{-1}  [A,-h_j^2\Delta_g]$ is the Poisson bracket $\big\{a,|\xi|^2_g\big\}=-H_{p}a$ therefore the left-hand side term equals
         $$ \big( \op{h_j} (H_{p}a)v_j,v_j\big)_{L^2} + O(h_j) $$ 
         Passing to the limit in the equality \eqref{CommutEq}, we finally get
         $$ \int_{T^*M} (H_{p}a) \, d \mu = 4\lambda \int_{T^*M} a \, d \mu $$
         which proves the claim.
\end{proof}
\begin{remark}
       If we were considering the semiclassical defect measure $\widehat{\mu}$ associated with quasimodes on a closed manifold $\widehat{M}$ then the transport equation would imply $\widehat{\phi}_t^*\widehat{\mu} = e^{2\lambda t} \widehat{\mu}$ where $\widehat{\phi}_t$ denotes the cogeodesic flow on $(\widehat{M},g)$.
\end{remark}

\subsection{Microlocal quasimodes} 

As in the previous section, to simplify notations we will write $(M,g)$ for the transversal manifold instead of $(M_0,g_0)$. Thus, let $(M,g)$ be an $m$-dimensional compact oriented manifold with smooth boundary and let $\gamma$ be a non-tangential geodesic. Once again, we embed $(M,g)$ in some closed manifold $(\widehat{M},g)$, extend $\gamma$ as a unit speed geodesic in $\widehat{M}$ and let $\eps>0$ be such that  $\gamma(t) \in \widehat{M} \smallsetminus M$ for $t \in [-2\eps,0) \cup (L,L+2\eps]$. We recall that $h=\tau^{-1}$ is our semiclassical parameter. After factorization of the operator 
\begin{align*}
     \Delta_g+(\tau+ i \lambda)^2 &= \tau^2 (h^2\Delta_g + (1+i \lambda h)^2) \\
     &= -\tau^2 \Big(\sqrt{-h^2\Delta_g}+1+i\lambda h\Big)\Big(\sqrt{-h^2\Delta_g}-1-i\lambda h\Big)
\end{align*}
it becomes clear that one has to seek $v_s$ such that 
    $$ \Big\|\Big(\sqrt{-h^2\Delta_g}-1-i\lambda h\Big)v_s\Big\|_{H^1_{\rm scl}(M)} = O(h^{K+2}), \quad s=h^{-1}+i\lambda. $$
In fact, we will construct  an $O(h^{\infty})$ quasimode.
First, we will need the following proposition from \cite{H3}, which is a global version of the microlocal canonical reduction of a principal type operator.
\begin{prop}
\label{HormNormProp}
      Let $I$ be a compact real interval and $\Gamma : I \to S^*\widehat{M}$ be a non-closed cogeodesic curve, let $\eps_m = (0,\dots,0,1) \in \R^m$, one can find a neighborhood $V$ of the segment $\Lambda=\{((x_1,0),\eps_m) :  x_1 \in I \} \subset \R^{2m}$ and a smooth canonical transformation $\varsigma : V \to \varsigma(V)$ from $V$ to an open neighborhood $\varsigma(V)$ of $\Gamma(I)$ such that
      $$ \varsigma((x_1,0),\eps_m) = \Gamma(x_1), \quad \varsigma^*(\sqrt{g^{-1}}-1)(x,\xi) = \xi_1. $$
\end{prop}
This is (a non-homogeneous version of) Proposition 26.1.6 in \cite{H3} applied to the symbol $a=\sqrt{g^{jk}(x)\xi_j\xi_k}-1$.  The proof follows H\"ormander's book and we  provide it in the non-homogeneous case.
Note that bicharacteristic curves of $a$ are cogeodesic in the cosphere bundle.
\begin{proof}
     Assume that $0 \in I$. Using local coordinates $(y,\eta)$ in a neighborhood $W_{\Gamma(0)}$ of $\Gamma(0) \in S^*\widehat{M}$, one can complete
      $$ \xi_1 = a(y,\eta) = \sqrt{g(y,\eta)}-1, \quad (y,\eta) \in W_{\Gamma(0)} $$
    into a system $(x,\xi)$ of symplectic coordinates near $(0,\eps_m)$ by the Darboux theorem.    
    The map
    \begin{align*}
         \chi : V_{(0,\eps_m)} &\to W_{\Gamma(0)} \\ (x,\xi) &\mapsto (y,\eta) 
    \end{align*}
    is a canonical relation from a neighborhood $V_{(0,\eps_m)}$ of $(0,\eps_m)$ in $T^*\R^m$ to a neighborhood $W_{\Gamma(0)}$ of $\Gamma(0)$ in $T^*\widehat{M}$ such that
         $$ \chi^*a =\xi_1, \quad \chi(0,\eps_m)=\Gamma(0). $$
     Note that 
         $$ \chi((x_1,0),\eps_m) = \Gamma(x_1) $$
     for all $x_1$ close to $0$ since both functions satisfy the same system of ordinary differential equations
     \begin{align}
     \label{Micro:SystODES}
          \frac{\partial \Gamma}{\partial x_1} &= H_a \circ \Gamma     \nonumber \\
          \frac{\partial \chi}{\partial x_1}((x_1,0),\eps_m) &= \{\xi_1,\chi((x_1,0),\eps_m)\} 
          = H_a \circ \chi((x_1,0),\eps_m)
     \end{align}
     with same initial condition.
     In order to prove the lemma, it suffices to extend $\chi$ to a neighborhood $V$ of the segment $\Lambda$, by taking the maximal solution $\varsigma$ of the system of differential equations
      \begin{align}
      \label{Micro:ExtODE}
      \begin{cases}
           \displaystyle
           \frac{\partial \varsigma}{\partial x_1} = H_a \circ \varsigma \\
           \varsigma(0,x',\xi) = \chi(0,x',\xi).
       \end{cases}
      \end{align}
      Note that $x_1 \to \varsigma(x,\xi)$ are bicharacteristic curves of $a$, hence cogeodesic curves in the cosphere bundle flowing through $V$. Therefore it follows from the assumption that $\Gamma$ is not closed that the extension $\varsigma$ is a diffeomorphism if $V$ is small enough. 
      Furthermore, it is a canonical transformation: if $(X,\Xi) = \varsigma^{-1}(y,\eta)$ then we have
      \begin{align*}
            H_a X_1 =1, \quad H_aX_j = 0 \quad \text{for }j=2,\dots,m,\quad H_a\Xi_j = 0 \quad \text{for } j=1,\dots,m,
      \end{align*}       
      from which we deduce by the Jacobi identity
      \begin{align*}
            H_a\{X_j,X_k\} = H_a\{\Xi_j,\Xi_k\} = H_a\{X_j,\Xi_k\}=0
      \end{align*}    
      and therefore these Poisson brackets are constant along integral curves of $H_p$, in particular
     \begin{align*}
            \{X_j,X_k\} =\{x_j,x_k\}=0, \;  \{\Xi_j,\Xi_k\}=\{\xi_j,\xi_k\} = 0, \; \{X_j,\Xi_k\}=\{x_j,\xi_k\} = -\delta_{jk}.
      \end{align*}  
      Since $\varsigma$ extends $\chi$ we have
            $$ \varsigma((x_1,0),\eps_m) = \chi((x_1,0),\eps_m) = \Gamma(x_1) $$
      and since $\varsigma$ is a solution of \eqref{Micro:ExtODE} we also get
            $$ a(\varsigma(x,\xi)) = a(\varsigma(0,x',\xi)) = a(\chi(0,x',\xi')) = \xi_1. $$
      This completes the construction of the canonical transformation $\varsigma$.
\end{proof}
Our choice for $\Gamma$ is the cogeodesic curve in $S^*\widehat{M}$ which projects on the geodesic $\gamma$ in $\widehat{M}$ and we take $I=[-\eps,L+\eps]$. The next step is the quantization of such a canonical transformation.
\begin{prop}
     Let $\Gamma : I \to S^*\widehat{M}$ be a non-closed cogeodesic curve, and let $\varsigma$ be the canonical transformation introduced in Proposition \ref{HormNormProp}. For all $\lambda \in \R$, there exist semiclassical Fourier integral operators $U_h \in I^0_{\rm scl}(\R^m \times \widehat{M},G')$, $V_h \in I^0_{\rm scl}(\widehat{M}\times \R^m,(G^{-1})')$ associated with the graphs $G$, resp. $G^{-1}$, of the canonical transformation $\varsigma$, resp. $\varsigma^{-1}$, such that $\WF(U_hV_h-1)$, $\WF(V_hU_h-1)$ do not intersect $\Gamma(I)$, resp. $\Lambda$, and such that
    \begin{align*}
         V_h \Big(\sqrt{-h^2\Delta_g}-1-i\lambda h\Big) U_h &= (hD_1-i\lambda h) + R_h + A_h
    \end{align*}
where $R_h \in h^{\infty}\Psi^{-\infty}_{\rm scl}(\widehat{M})$ and $A_h \in \Psi^{0}_{\rm scl}(\widehat{M})$ is such that $\WF(A_h) \cap \Lambda = \varnothing$. 
\end{prop}
\begin{proof}
   There exist%
   \footnote{One chooses $U_h$ to be non-characteristic near $\Lambda \times \Gamma(I)$ and  the construction of $V_h$ is the standard construction of a parametrix (see Remark on bottom of page 27 after Definition 25.3.4 in \cite{H3} for the classical case, Theorem 11.5 in \cite{Zworski} for the semiclassical case).} semiclassical Fourier integral operators $U_h \in I^0_{\rm scl}(\R^m \times \widehat{M},G')$, $V_h \in I^0_{\rm scl}(\widehat{M}\times \R^m,(G^{-1})')$   such that $\WF(U^0_hV^0_h-1)$, $\WF(V^0_hU^0_h-1)$ do not intersect $\Gamma(I)$, resp. $\Lambda$. 
 By Egorov's theorem one has
       \begin{align*}
         V^0_h \Big(\sqrt{-h^2\Delta_g}-1\Big) U^0_h = V^0_hU^0_h hD_1 + hR^0_h
    \end{align*}
where $R^0_h \in \Psi^0_{\rm scl}(\widehat{M})$. It remains to improve the remainder 
$R_h^0$, and this can be done by further conjugation by two elliptic pseudodifferential operators $U_h^1,V^1_h \in \Psi^0_{\rm scl}(\R^m)$ such that $V^1_h U^1_h-1 \in  h^{\infty}\Psi^{-\infty}_{\rm scl}(\R^m)$  and 
    $$ [hD_ 1,U_h^1] + hR^0_hU_h^1  \in h^{\infty} \Psi^{-\infty}_{\rm scl}(\R^m).$$
    This can be done by choosing $U_h^1=\op{h} a$ where $a \sim \sum_{j \geq 0} h^{j}a_j$ is the asymptotic sum of a sequence of symbols $(a_j)_{j \geq 0}$ satisfying the recursive equations
        $$ \frac{1}{i}\partial_{x_1} a_j + r a_j = -r_{j-1} $$
where $r$ is the principal symbol of $R^0_h$, $r_{-1}=0$,  and $r_{j-1}$ is a principal symbol of the operator
     $$ R_j =h^{-j-2} \Big([hD_1,\op{h}(a_0+\cdots+h^j a_j)] + hR^0_h \op{h}(a_0+\cdots+h^j a_j)\Big) \in \Psi^{-j-1}_{\rm scl}. $$
     This sequence of equations can be explicitly solved and the solutions
     \begin{align*}
          a_0 &= \exp\bigg(\int_0^{x_1} r(y_1,x') \, dy_1\bigg) \\
          a_j &= -i a_0(x,\xi) \bigg(\int_0^{x_1} r_{j-1}(y_1,x_1) a_0^{-1}(y_1,x') \, dy_1\bigg)
     \end{align*}
     are symbols of order $j$.
    
    Taking $U_h=U_h^0U_h^1$ and $V_h=V_h^1V_h^0$ we finally get       
    \begin{align*}
         V_h \Big(\sqrt{-h^2\Delta_g}-1-i\lambda h\Big) U_h =  (hD_1-i\lambda h) + A_h+ R_h
    \end{align*}
    with 
    \begin{align*}
         R_h=V_h^1 ( [hD_ 1,U_h^1] + hR^0_hU^1_h)+(V^1_hU^1_h-1)(hD_1)  \in h^{\infty} \Psi^{-\infty}_{\rm scl}(\R^m)
    \end{align*}
    and where the remainder
    \begin{align*}
         A_h=V_h^1(V_h^0U_h^0-1)(hD_1)U^1_h+i\lambda h (1-V_hU_h)
    \end{align*}
    has a semiclassical wavefront set which does not meet $\Lambda$ because of the wavefront set properties of $U^0_h$ and $V^0_h$.
\end{proof}

Having reduced the operator, it is now easy to construct quasimodes for the simple normal form
$hD_1-i\lambda h$; in fact, we may as well choose a solution of the equation $(\partial_{1}+\lambda)v=0$, and take as our quasimode the function
\begin{align}
\label{Microquasimode}
     v_s = U_h(H(x_1) e^{-\lambda x_1}w_h(x')), \quad \|w_h\|_{L^2(\R^{m-1})}=O(1)
\end{align}
where $x=(x_1,x') \in \R^m$, $w_h$ is smooth and where $H$ is a smooth cutoff function supported in $[-2,+\infty)$ which equals $1$ on $[-1,+\infty)$ so that $H(x_1) e^{-\lambda x_1}w_h(x')$ is an $L^2$ function.
That this could indeed be a possible quasimode is a consequence of the following relation 
    \begin{multline*}
         U_hV_h \Big(\sqrt{-h^2\Delta_g}-1-i\lambda h\Big) v_s = U_h\underbrace{(hD_1-i\lambda h)(H(x_1)e^{-\lambda x_1}w_h(x'))}_{=-i (h\partial_{x_1}H)e^{-\lambda x_1}w_h } \\ + U_hR_h(H e^{-\lambda x_1}w_h(x')) + U_hA_h(H e^{-\lambda x_1}w_h(x'))
    \end{multline*}
which leads to the estimate    
     \begin{multline*}
        \Big\|(\op{h}\chi) \Big(\sqrt{-h^2\Delta_g}-1-i\lambda h\Big) v_s \Big\|_{H^1_{\rm scl}}  \\ \leq \|(\op{h}\chi) U_hA_h(H e^{-\lambda x_1}w_h(x'))\|_{H^1_{\rm scl}} + O(h^{\infty})
    \end{multline*}
when $\chi \in C^{\infty}_c(V)$ is a symbol which equals $1$ close to $\Gamma(I)$. It follows from      
 $$ \WF(U_hA_h(H e^{-\lambda x_1}w_h(x')) \cap \Gamma(I) \subset 
       (\varsigma(\WF(A_h)))' \cap \Gamma(I) = \varnothing $$
that we have
    \begin{align*}
         \|(\op{h}\chi) U_hA_h(He^{-\lambda x_1}w_h(x'))\|_{H^1_{\rm scl}} = O(h^{\infty}).
    \end{align*}

Since $\chi$ is localized in a neighbourhood of the cogeodesic $\Gamma$, we need an additional estimate away from $\Gamma(I)$; in order to have such an estimate, we must impose 
on our quasimode that its semiclassical wave front set be contained in $\Gamma(I)$. This means that we require
\begin{align}
\label{WFcond}
     \WF(w_h) = \{(0,\eps_m)\}.
\end{align}
\begin{lemma}
    The semiclassical wave front set of the quasimode $v_s$ given by \eqref{Microquasimode} with the microlocal constraint  \eqref{WFcond} is contained in the cogeodesic curve $\Gamma(I)$:
          $$ \WF(v_{h^{-1}+i\lambda}) \subset \Gamma(I). $$
\end{lemma}
\begin{proof}
     Let $\psi$ be a cutoff function. From the remark just after the definition of wavefront sets involving the semiclassical Fourier transform and from the example 4.1.2, we deduce that the semiclassical wave front set of $\psi(x) e^{-\lambda x_1}w_h(x')$ is contained in the line $\Lambda=\{(x_1,0,\eps_m) \in \R^{2m} : x_1 \in \R \}$. The lemma follows then from Lemma \ref{PropagationLemma} since the line $\Lambda$ is mapped into the cogeodesic $\Gamma$ by the canonical transformation $\varsigma$. 
\end{proof}
From the Lemma, we have $\WF\big(\big(\sqrt{-h^2\Delta_g}-1-i\lambda h\big)v_{h^{-1}+i\lambda}\big) \subset \Gamma(I)$ and since $1-\chi$ is supported away from $\Gamma(I)$, we deduce
    \begin{align*}
        \Big\| (1 - \op{h}\chi) \Big(\sqrt{-h^2\Delta_g}-1-i\lambda h\Big) v_s \Big\|_{H^1_{\rm scl}}  = O(h^{\infty})
    \end{align*}
Together with the previous estimate, this proves that $v_s$ is a quasimode.

Having constructed our quasimode, we proceed to the study of the corresponding semiclassical measure $\mu$. Let $a \in C^{\infty}_c(T^*M^{\rm int})$, we have
    \begin{multline}
    \label{Transfersemiclassical}
          \int_{M} (\op{h} a) v_{h^{-1}+i\lambda} \, \overline{v_{h^{-1}+i\lambda}} \, dV \\ = \int_{\R^{m}} U_h^*(\op{h} a) U_h (H(x_1) e^{-\lambda x_1} w_h(x')) \, \overline{H(x_1) e^{-\lambda x_1} w_h(x')} \, d x_1 \, dx'.
    \end{multline}
By Egorov theorem, the conjugated operator has a simple principal expression 
    \begin{align*}
          U_h^*(\op{h} a) U_h = \op{h}(\chi \varsigma^*a ) + hR_h
    \end{align*}
where $\chi$ is the principal symbol of $U_h^*U_h$ and $R_h \in \Psi^0_{\rm scl}$. We choose our function $w_h$ to be a wave packet
        $$ w_h =(\pi h)^{-\frac{m-1}{4}} e^{-\frac{1}{2h}|x'|^2+\frac{i}{h}x_m} $$
and in the construction of $U_h$, one can take $\chi$ to be one on $\Gamma(I)$. 
\begin{lemma}
   The semiclassical measure associated to the family 
        $$ \widetilde{v}_s = (\pi h)^{-\frac{m-1}{4}} H(x_1) e^{-\lambda x_1} e^{-\frac{1}{2h}|x'|^2+\frac{i}{h}x_m} $$
    is $H^2(x_1) e^{-2\lambda x_1} \, dx_1 \otimes \delta_{x'=0,\xi=\eps_m}$.
\end{lemma}
\begin{proof}
     Let $a \in C^{\infty}_c(\R^{2m})$, we have 
     \begin{multline*}
           \big( (\op{h}a)\widetilde{v}_s,\widetilde{v}_s \big)_{L^2(\R^m)} =
            2^{-m} (\pi h)^{-\frac{3m-1}{2}} \iiint H(x_1)H(y_1) e^{-\lambda (x_1+y_1)} \\
            \times e^{-\frac{1}{2h}(|x'|^2+|y'|^2)} e^{\frac{i}{h}(x-y) \cdot (\xi-\eps_m)} a\bigg(\frac{x+y}{2},\xi\bigg)  \, d x \, dy  \, d \xi.  
     \end{multline*}
     We can take $(x+y)/2$ and $(x-y)/2$ as new coordinates and after integration, we get 
      \begin{multline*}
           \big( (\op{h}a)\widetilde{v}_s,\widetilde{v}_s \big)_{L^2(\R^m)} =
            (\pi h)^{-m}\iiint e^{-2\lambda x_1} H(x_1+y_1) H(x_1-y_1) \\ \times 
            e^{\frac{2i}{h} y_1 \xi_1} e^{-\frac{1}{h}(|x'|^2+|\xi'-\eps'_m|^2)}  a(x_1,x',\xi_1,\xi')  \, d x \, dy_1 \, d \xi.  
     \end{multline*}
     We let $h$ tend to $0$ and obtain
         $$  \lim_{h \to 0} \big( (\op{h}a)\widetilde{v}_s,\widetilde{v}_s \big)_{L^2(\R^m)} = \int_{-\infty}^{\infty} H^2(x_1) a(x_1,0,\eps_m) e^{-2\lambda x_1} \, dx_1 $$
     which completes the proof.
\end{proof}

Using the lemma, Egorov's theorem and passing to the limit in \eqref{Transfersemiclassical}, we get 
\begin{align*}
     \int_{T^*M} a \, d \mu &= \int_{-\eps}^{L+\eps} \chi(x_1,0,\eps_m) \varsigma^*a(x_1,0,\eps_m) e^{-2\lambda x_1} \, d x_1 \\ &=  \int_{0}^{L} a(\Gamma(x_1)) e^{-2\lambda x_1} \, d x_1, \quad a \in C^{\infty}_c(M^{\rm int})
\end{align*}
since $\chi$ equals one on $\Gamma(I)$. One can sum up our construction in the following theorem.
\begin{thm}
      For any non-tangential geodesic on a compact Riemannian manifold $(M_0,g_0)$ with boundary, there exists a family of quasimodes $(v_{h^{-1}+i\lambda})_{h \in (0,1)}$ such that
      \begin{align*}
            (h^2\Delta_{g_0}+(1+i\lambda h)^2)v_{h^{-1}+i\lambda} &= O(h^{\infty})  \\
            \|v_{h^{-1}+i\lambda}\|_{L^2(M_0)} &= O(1) 
      \end{align*}       
 with semiclassical wave front set contained in the cogeodesic $\Gamma$ projecting on $\gamma$ and with associated semiclassical measure $\mu$ on $M_0^{\rm int}$ given by
\begin{align*}
          \int_{T^*M_0} a \, d \mu =  \int_{0}^{L} a(\Gamma(x_1)) e^{-\lambda x_1} \, d x_1.
\end{align*}
\end{thm}
From this alternate construction, one can also deduce Theorem \ref{claim4}.

\section{Calder{\'o}n problem in a cylinder} \label{sec_bcmethod_first}

In this section we will prove Theorems \ref{thm_calderon_cylinder_reconstruction} and \ref{thm_calderon_cylinder_uniqueness}, which consider an inverse problem in the infinite cylinder $T = \mR \times M_0$ with metric $g = e \oplus g_0$. Here $(M_0,g_0)$ is a compact oriented $m$-dimensional manifold with smooth boundary, $m \geq 2$. We write $(t,x)$ for coordinates on $T$ where $t$ is the Euclidean coordinate and $x$ are coordinates on $M_0$. The Laplace-Beltrami operator in $T$ is given by 
$$
\Delta = \Delta_g = \partial_t^2 + \Delta_{g_0}.
$$
We consider more generally the Schr\"odinger operator on $T$, 
$$
-\Delta + q_0 = -\partial_t^2 - \Delta_{g_0} + q_0
$$
where $q_0 \in C^{\infty}(M_0)$Ê   is real valued. It will be crucial that the coefficients $g_0$ and $q_0$Ê   are independent of the $t$ variable.

The first point is to set up boundary measurements related to the Dirichlet problem 
$$
(-\partial_t^2 - \Delta_{g_0} + q_0) u = 0 \text{ in } T, \quad u = h \text{ on } \partial T.
$$
The spectral properties of the Schr\"odinger operator in the infinite cylinder are different from those on a compact manifold because of the presence of continuous spectrum. Let $\lambda_1 \leq \lambda_2 \leq \cdots \to \infty$ be the Dirichlet eigenvalues of $-\Delta_{g_0}+q_0$Ê   in $M_0$, write $\text{Spec}(-\Delta_{g_0}+q_0) = \{ \lambda_1, \lambda_2, \ldots \}$, and let $\{ \phi_l \}_{l=1}^{\infty}$ be an orthonormal basis of $L^2(M_0)$Ê   consisting of eigenfunctions which satisfy $(-\Delta_{g_0} + q_0) \phi_l = \lambda_l \phi_l$ in $M_0$, $\phi_l \in H^1_0(M_0)$.

We next define certain function spaces. Let $L^2(T) = L^2(T,dV)$ be the standard $L^2$ space in $T$, and let $H^s(T)$ be the corresponding $L^2$ Sobolev spaces. Since $M_0$ is compact, we define 
\begin{gather*}
H^s_{\text{loc}}(T) = \{ f \,;\, f \in H^s([-R,R] \times M_0) \text{ for any $R > 0$} \}.
\end{gather*}
Writing $\br{t} = (1+t^2)^{1/2}$, we introduce for $s \geq 0$ the weighted spaces 
\begin{gather*}
L^2_{\delta}(T) = \{ f \in L^2_{\text{loc}}(T) \,;\, \br{t}^{\delta} f \in L^2(T) \}, \\
H^s_{\delta}(T) = \{ f \in H^s_{\text{loc}}(T) \,;\, \br{t}^{\delta} f \in H^s(T) \}, \\
H^1_{\delta,0}(T) = \{ f \in H^1_{\delta}(T) \,;\, f|_{\partial T} = 0 \}.
\end{gather*}
Also, $H^1_0(T) = \{ f \in H^1(T) \,;\, f|_{\partial T} = 0 \}$. We define, in the $L^2(T)$ duality, 
$$
H^{-1}(T) = (H^1_0(T))^*.
$$

If $s \geq 1/2$ define the abstract trace spaces 
\begin{gather*}
H^s(\partial T) = H^{s+1/2}(T)/(H^{s+1/2}(T) \cap H^1_0(T)), \\
H^s_{\delta}(\partial T) = H^{s+1/2}_{\delta}(T)/(H^{s+1/2}_{\delta}(T) \cap H^1_0(T)).
\end{gather*}
Since $\partial M_0$ is smooth and compact, these spaces can also be identified with standard weighted Sobolev spaces on $\partial T$.

We will see the following facts:
\begin{itemize}
\item 
$-\Delta + q_0$ with domain $H^2(T) \cap H^1_0(T)$ is self-adjoint on $L^2(T)$,
\item 
the spectrum of $-\Delta + q_0$ is $[\lambda_1,\infty)$,
\item 
if $\lambda \in \mC \smallsetminus [\lambda_1,\infty)$ then for any $\delta \in \mR$ 
$$
(-\Delta + q_0 - \lambda)^{-1}: L^2_{\delta}(T) \to \{ u \in H^2_{\delta}(T) \,;\, u|_{\partial T} = 0 \},
$$
\item 
if $\lambda \in [\lambda_1,\infty)$ and $\lambda \notin \text{Spec}(-\Delta_{g_0}+q_0)$ then for any $\delta > 1/2$ the following limiting absorption principle holds:
$$
(-\Delta + q_0 - \lambda-i0)^{-1}: L^2_{\delta}(T) \to \{ u \in H^2_{-\delta}(T) \,;\, u|_{\partial T} = 0 \}.
$$
\end{itemize}
The case of thresholds, where $\lambda = \lambda_l$, is special and will not be considered here.

In this section we will assume that $\lambda$ is not in the spectrum $[\lambda_1,\infty)$ (the general case $\lambda \in \mC \setminus \{ \lambda_1, \lambda_2, \ldots\}$ is considered in the next section). The following proposition shows that there is a well defined DN map $\Lambda_{g_0,q_0}^T(\lambda)$ related to the operator $-\Delta+q_0-\lambda$ in the cylinder $T$.

\begin{prop} \label{prop_wp_dirichlet_cylinder}
If $\lambda \in \mC \setminus [\lambda_1,\infty)$, then for any $f \in H^{3/2}(\partial T)$ there is a unique solution $u \in H^2(T)$ of the equation 
$$
(-\Delta+q_0-\lambda)u = 0 \ \ \text{in } T, \qquad u|_{\partial T} = f.
$$
If $f \in C^{\infty}_c(\partial T)$â then $u \in C^{\infty}(T)$ and there is a linear map 
$$
\Lambda_{g_0,q_0}^T(\lambda): C^{\infty}_c(\partial T) \to C^{\infty}(\partial T), \ \ f \mapsto \partial_{\nu} u|_{\partial T}.
$$
For any $\delta \in \mR$, this map extends as a bounded linear map 
$$
\Lambda_{g_0,q_0}^T(\lambda): H^{3/2}_{\delta}(\partial T) \to H^{1/2}_{\delta}(\partial T).
$$
\end{prop}

The first easy observation is that one has unique solvability of the Dirichlet problem in $T$ for certain frequencies by the usual Lax-Milgram type argument. The key point is that $\partial T$ is compact in the direction transverse to $t$, so it has a Poincar\'e inequality in the standard Sobolev spaces.

\begin{lemma}
Let $\lambda$ satisfy $-\infty < \lambda < \lambda_1$. For any $F \in H^{-1}(T)$ the equation $(-\Delta + q_0 - \lambda)u = F$ in $T$ has a unique solution $u \in H^1_0(T)$, and one has $\norm{u}_{H^1(T)} \leq C \norm{F}_{H^{-1}(T)}$.
\end{lemma}
\begin{proof}
Write $d$ and $d_x$ for the exterior differentials in $T$ and $M_0$, respectively. We consider the bilinear form 
$$
B(u,v) = \int_T ( \langle d u, d \bar{v} \rangle + q_0 u \bar{v} - \lambda u \bar{v}) \,dV, \quad u, v \in H^1_0(T).
$$
This is a bounded symmetric sesquilinear form on $H^1_0(T)$ and satisfies 
$$
B(u,u) = \int_{-\infty}^{\infty} \int_{M_0} (\abs{\partial_t u}^2 + \abs{d_x u}^2 + q_0 \abs{u}^2 - \lambda \abs{u}^2) \,dV_{g_0} \,dt.
$$
We note the inequality 
$$
\int_{M_0} (\abs{d_x v}^2 + q_0 \abs{v}^2) \,dV_{g_0} \geq \lambda_1 \int_{M_0} \abs{v}^2 \,dV_{g_0}, \quad v \in H^1_0(M_0).
$$
This shows that $B(u,u) \geq \eps \norm{du}_{L^2(T)}^2 + (\lambda_1(1-\eps) - \eps \norm{q_0}_{L^{\infty}} - \lambda) \norm{u}_{L^2(T)}^2$ for $0 < \eps < 1$. If $\eps$ is sufficiently small we see that $B$ is coercive and the Riesz representation theorem shows the existence of a unique solution of $(-\Delta + q_0 - \lambda) u = F$ in $T$.
\end{proof}

Elliptic regularity for the previous problem, even in weighted spaces, can be proved by a Fourier analysis argument.

\begin{lemma} \label{lemma_wpcylinder1}
Let $\lambda \in \mC \smallsetminus [\lambda_1,\infty)$ and let $\delta \in \mR$. For any $F \in L^2_{\delta}(T)$ there is a unique solution $u \in H^1_{\delta,0}(T)$ of the equation $(-\Delta + q_0 - \lambda) u = F$ in $T$. Further, $u \in H^2_{\delta}(T)$ and $\norm{u}_{H^2_{\delta}(T)} \leq C_{\delta,\lambda} \norm{F}_{L^2_{\delta}(T)}$.
\end{lemma}
\begin{proof}
Write $\widetilde{v}(t,l) = (v(t,\,\cdot\,), \phi_l)_{L^2(M_0)}$ for the partial Fourier coefficients. If $F \in L^2_{\delta}(T)$ we note that 
\begin{align*}
\norm{F}_{L^2_{\delta}(T)}^2 &= \int_{-\infty}^{\infty} \int_{M_0} \br{t}^{2\delta} \abs{F(t,x)}^2 \,dV_{g_0} \,dt = \int_{-\infty}^{\infty} \br{t}^{2\delta} \sum_{l=1}^{\infty} \abs{\widetilde{F}(t,l)}^2 \,dt \\
 &= \sum_{l=1}^{\infty} \norm{\widetilde{F}(\,\cdot\,,l)}_{L^2_{\delta}(\mR)}^2.
\end{align*}
This implies that $\widetilde{F}(\,\cdot\,,l) \in L^2_{\delta}(\mR)$ for all $l$, and the Fourier transform satisfies $\widehat{F}(\,\cdot\,,l) \in H^{\delta}(\mR)$. Formally, the equation $(-\Delta + q_0 - \lambda) u = F$ reduces to a system of ODEs:
$$
(-\partial_t^2 + \lambda_l - \lambda) \widetilde{u}(\,\cdot\,,l) = \widetilde{F}(\,\cdot\,,l) \ \ \text{on $\mR$, for $l = 1,2,\ldots$.}
$$
By taking Fourier transforms in $t$ (with dual variable $\eta$), we obtain 
\begin{equation} \label{tildeutl}
\widetilde{u}(t,l) = \mF_{\eta}^{-1} \left\{ \frac{1}{\eta^2 + \lambda_l - \lambda} \widehat{F}(\eta,l) \right\}.
\end{equation}

Uniqueness follows immediately since if $u \in H^1_{\delta,0}(T)$ solves the equation $(-\Delta + q_0 - \lambda)u = 0$Ê   in $T$ for some $\delta \in \mR$, then $\widetilde{u}(\,\cdot\,,l) \in L^2_{\delta}(\mR)$ and by taking Fourier transforms $(\eta^2 + \lambda_l - \lambda) \widehat{u}(\eta,l) = 0$ for $\eta \in \mR$ and for all $l$. Here $\eta^2 + \lambda_l - \lambda$ is never zero using that $\lambda \in \mC \smallsetminus [\lambda_1,\infty)$, so $\widetilde{u}(\,\cdot\,,l) = 0$Ê   for all $l$Ê   and $u = 0$.

We move to existence and let $F \in L^2_{\delta}(T)$. If $\widetilde{u}(t,l)$ is defined by \eqref{tildeutl}, then for $k \geq \abs{\delta}$ we have 
$$
\norm{\widetilde{u}(\,\cdot\,,l)}_{L^2_{\delta}(\mR)} = \norm{\widehat{u}(\,\cdot\,,l)}_{H^{\delta}(\mR)} \leq \norm{(\eta^2 + \lambda_l - \lambda)^{-1}}_{W^{k,\infty}(\mR)} \norm{\widehat{F}(\,\cdot\,,l)}_{H^{\delta}(\mR)}.
$$
We need to estimate the $W^{k,\infty}$ norm uniformly in $l$, using the condition $\lambda \in \mC \smallsetminus [\lambda_1,\infty)$. Note that 
\begin{equation} \label{eta_partialfraction}
(\eta^2 + \lambda_l - \lambda)^{-1} = \frac{1}{2 z_l} \left( \frac{1}{\eta - z_l} - \frac{1}{\eta + z_l} \right)
\end{equation}
where $z_l^2 = \lambda - \lambda_l = \alpha - (\lambda_l - \lambda_1) + i\beta$ and $\alpha = \re(\lambda) - \lambda_1, \beta = \im(\lambda)$ (we assume $z_l \in \{ z \,;\, \re(z) > 0, \im(z) \neq 0 \} \cup \{ ir \,;\, r > 0 \}$). Write 
$$
f(t) = f_{\alpha,\beta}(t) = \im((\alpha - t + i\beta)^{1/2}), \quad t \geq 0.
$$
When $\alpha \leq 0$Ê   we have $\abs{f(t)} \geq \abs{f(0)}$ for $t \geq 0$, and when $\alpha > 0$Ê   we have $\abs{f(t)} \geq \abs{f(\alpha)}$ for $t \geq \alpha$ and $\abs{f(t)} \geq c_{\alpha,\beta}$ for $0 \leq t \leq \alpha$. This shows that 
$$
\abs{\im(z_l)} \geq c_{\alpha,\beta} > 0.
$$
Since 
\begin{equation} \label{eta_function_derivative}
\left( \frac{\partial}{\partial \eta} \right)^{j} (\eta^2 + \lambda_l - \lambda)^{-1} = \frac{C_j}{2 z_l} \left( \frac{1}{(\eta - z_l)^{j+1}} - \frac{1}{(\eta + z_l)^{j+1}} \right),
\end{equation}
and since $\abs{z_l}^2 = \abs{\lambda-\lambda_l} \geq \abs{\im(\lambda)}$ if $\im(\lambda) \neq 0$ and $\abs{\lambda-\lambda_l} \geq \abs{\re(\lambda)-\lambda_1}$ when $\im(\lambda) = 0$, we have 
$$
\norm{(\eta^2 + \lambda_l - \lambda)^{-1}}_{W^{k,\infty}(\mR)} \leq C_{k,\lambda}.
$$
Thus
$$
\norm{\widetilde{u}(\,\cdot\,,l)}_{L^2_{\delta}(\mR)} \leq C_{\delta,\lambda} \norm{\widetilde{F}(\,\cdot\,,l)}_{L^2_{\delta}(\mR)}.
$$

We define 
$$
u_N(t,x) = \sum_{l=1}^N \widetilde{u}(t,l) \phi_l(x).
$$
It follows that when $M \leq N$, 
$$
\norm{u_M - u_N}_{L^2_{\delta}(T)}^2 = \sum_{l=M+1}^N \norm{\widetilde{u}(\,\cdot\,,l)}_{L^2_{\delta}(\mR)}^2 \leq C \sum_{l=M+1}^N \norm{\widetilde{F}(\,\cdot\,,l)}_{L^2_{\delta}(\mR)}^2.
$$
Thus $(u_N)$ is a Cauchy sequence in $L^2_{\delta}(T)$ and converges to some $u \in L^2_{\delta}(T)$. Since 
$$
(-\Delta + q_0 - \lambda) u_N(t,x) = \sum_{l=1}^N \widetilde{F}(t,l) \phi_l(x),
$$
we obtain that $u$ is a distributional solution of $(-\Delta + q_0 - \lambda) u = F$ in $T$.

We next show that $u \in H^1_{\delta}(T)$. The expression for $\widetilde{u}(t,l)$ together with \eqref{eta_partialfraction} implies that 
$$
\norm{\partial_t \widetilde{u}(\,\cdot\,,l)}_{L^2_{\delta}} + \abs{\lambda_l-\lambda}^{1/2 } \norm{\widetilde{u}(\,\cdot\,,l)}_{L^2_{\delta}} \leq C \norm{\widetilde{F}(\,\cdot\,,l)}_{L^2_{\delta}}.
$$
Thus $\norm{\partial_t u}_{L^2_{\delta}} \leq C \norm{F}_{L^2_{\delta}}$ and  
\begin{align*}
\norm{\nabla_{g_0} u_N(t)}_{L^2(M_0)}^2 &= (-\Delta_{g_0} u_N(t), u_N(t))_{L^2(M_0)} \\
 &= \sum_{l=1}^N (\lambda_l-\lambda) \abs{\widetilde{u}(t,l)}^2 + ((\lambda-q_0) u_N(t), u_N(t))_{L^2(M_0)}.
\end{align*}
Consequently $\norm{\nabla_{g_0} u}_{L^2_{\delta}(T)} \leq C \norm{F}_{L^2_{\delta}(T)}$, and then also $u \in H^1_{\delta,0}(T)$.

Finally, we check that $u \in H^2_{\delta}(T)$. Note that 
$$
(\lambda_l - \lambda) \widetilde{u}(t,l) = - \mF_{\eta}^{-1} \left\{ \frac{z_l^2}{\eta^2 - z_l^2} \widehat{F}(\eta,l) \right\}
$$
and so, for $k \geq \abs{\delta}$,  
$$
\norm{(\lambda_l - \lambda) \widetilde{u}(\,\cdot\,,l)}_{L^2_{\delta}} \leq \norm{z_l^2/(\eta^2-z_l^2)}_{W^{k,\infty}} \norm{\widetilde{F}(\,\cdot\,,l)}_{L^2_{\delta}}.
$$
We use \eqref{eta_function_derivative} to write 
\begin{align*}
\left( \frac{\partial}{\partial \eta} \right)^{j} \frac{z_l^2}{\eta^2-z_l^2} &= \frac{C_j z_l}{2} \frac{(\eta+z_l)^{j+1} - (\eta-z_l)^{j+1}}{(\eta^2-z_l^2)^{j+1}} \\
 &= C_j \frac{z_l^2}{\eta^2 - z_l^2} \sum_{a=0}^j \frac{C_{ja} \eta^{j-a} z_l^a}{(\eta^2-z_l^2)^j}.
\end{align*}
We choose $L = L(\lambda, g_0)$ so large that for $\re(z_l^2) < 0$ for $l \geq L$. Then $\abs{\eta^2-z_l^2} \geq \abs{z_l}^2$ for $l \geq L$, and also $\abs{z_l^2/(\eta^2-z_l^2)} \leq C_{\lambda,g_0}$ for $l=1, \ldots, L-1$. Also $\abs{\eta/(\eta^2-z_l^2)}, \abs{z_l/(\eta^2-z_l^2)} \leq C_{\lambda,g_0}$. It follows that 
$$
\norm{(\lambda_l - \lambda) \widetilde{u}(\,\cdot\,,l)}_{L^2_{\delta}} \leq C \norm{\widetilde{F}(\,\cdot\,,l)}_{L^2_{\delta}}.
$$
Therefore $(-\Delta_{g_0} + q_0 - \lambda) u \in L^2_{\delta}(T)$. By elliptic regularity 
$$
\norm{\nabla_{g_0}^2 u(t,\,\cdot \,)}_{L^2(M_0)} \leq C(\norm{u(t,\,\cdot \,)}_{L^2(M_0)} + \norm{\Delta_{g_0} u(t,\,\cdot \,)}_{L^2(M_0)})
$$
for almost every $t$. Consequently $\nabla_{g_0}^2 u \in L^2_{\delta}(T)$, which implies that also $\partial_t^2 u \in L^2_{\delta}(T)$. By a short argument we obtain $\nabla_{g_0} \partial_t u \in L^2_{\delta}(T)$ and the proof is finished.
\end{proof}

The previous lemma also implies self-adjointness.

\begin{lemma}
If $q_0 \in C^{\infty}(M_0)$Ê   is real valued, then the operator $-\Delta + q_0$ with domain $H^2(T) \cap H^1_0(T)$ is self-adjoint on $L^2(T)$.
\end{lemma}
\begin{proof}
The operator $-\Delta + q_0$ with this domain is densely defined and symmetric, and by Lemma \ref{lemma_wpcylinder1} the range of $-\Delta + q_0 \pm i$ is all of $L^2(T)$.
\end{proof}

Proposition \ref{prop_wp_dirichlet_cylinder} is a straightforward consequence of Lemma \ref{lemma_wpcylinder1} and elliptic regularity (see also the proof of Proposition \ref{prop_wp_dirichlet_cylinder2} below). The key point in the proof of the uniqueness result, Theorem \ref{thm_calderon_cylinder_reconstruction}, is the following connection between the DN map for Schr\"odinger operator $-\partial_t^2 - \Delta_{g_0} + q_0$ in $T$ and the DN map for the transversal Schr\"odinger operator $-\Delta_{g_0} + q_0$ in $M_0$. We define the transversal DN map at energy $\mu \in \mC \setminus \{ \lambda_1,\lambda_2,\ldots \}$ as 
$$
\Lambda_{g_0,q_0}^{M_0}(\mu): H^{3/2}(\partial M_0) \to H^{1/2}(\partial M_0), \ \ h \mapsto \partial_{\nu} v_h|_{\partial M_0}
$$
where $v_h$ is the unique solution of the Dirichlet problem 
$$
(-\Delta_{g_0} + q_0 - \mu)v_h = 0 \ \ \text{in } M_0, \qquad v_h|_{\partial M_0} = h.
$$

\begin{prop} \label{lemma_dnmap_transversal_cylinder}
If $\lambda \in \mC \setminus [\lambda_1,\infty)$ and if $k \in \mR$, then 
$$
\Lambda_{g_0,q_0}^{M_0}(\lambda-k^2) h = e^{-ikt} \Lambda_{g_0,q_0}^T(\lambda)(e^{ikt} h|_{\partial T}).
$$
In particular, the expression on the right is independent of the $t$ variable.
\end{prop}
\begin{proof}
Let $h \in H^{3/2}(\partial M_0)$, and let $v_h \in H^2(M_0)$ solve 
$$
(-\Delta+q_0-(\lambda-k^2)) v_h = 0 \ \ \text{in } M_0, \qquad v_h|_{\partial M_0} = h.
$$
Note that since $\lambda \notin [\lambda_1,\infty)$, the number $\lambda-k^2$ is not a Dirichlet eigenvalue of $-\Delta_{g_0}+q_0$ and there is a unique solution $v_h$. Define 
$$
f(t,x) = e^{ikt} h(x).
$$
Since $k$ is real, we have $f \in H^{3/2}_{\delta}(\partial T)$ for any $\delta < -1/2$. The function $u(t,x) = e^{ikt} v_h(x)$ is in $H^2_{\delta}(T)$ and solves 
$$
(-\partial_t^2 - \Delta_{g_0} + q_0 - \lambda) u = 0 \ \ \text{in } T, \qquad u|_{\partial T} = f.
$$
Thus 
$$
\Lambda_{g_0,q_0}^T(\lambda) f = \partial_{\nu} u|_{\partial T} = e^{ikt} (\partial_{\nu} v_h|_{\partial M_0}) = e^{ikt} \Lambda_{g_0,q_0}^{M_0}(\lambda-k^2) h.
$$
This proves the result.
\end{proof}

We can now prove Theorems \ref{thm_calderon_cylinder_reconstruction} and \ref{thm_calderon_cylinder_uniqueness}, showing that the DN map $\Lambda_{g_0,q_0}^T(\lambda)$ at a fixed energy $\lambda \notin [\lambda_1,\infty)$ uniquely determines the metric $g_0$ up to isometry and also the potential $q_0$.

\begin{proof}[Proof of Theorem \ref{thm_calderon_cylinder_reconstruction}]
Suppose one is given the manifold $\partial T = \mR \times \partial M_0$ and the map $\Lambda_{g_0,q_0}^T(\lambda): C^{\infty}_c(\partial T) \to C^{\infty}(\partial T)$ for some fixed $\lambda \in \mC \setminus [\lambda_1,\infty)$. We may assume that $\partial M_0$ is known. Since $C^{\infty}_c(\partial T)$ is dense in $H^{3/2}_{\delta}(\partial T)$ for all $\delta$, we also know the map 
$$
\Lambda_{g_0,q_0}^T(\lambda): H^{3/2}_{\delta}(\partial T) \to H^{1/2}_{\delta}(\partial T)
$$
for all $\delta$.

Since $e^{ikt} h \in H^{3/2}_{\delta}(\partial T)$ whenever $k \in \mR$ and $\delta < -1/2$, we may compute the map 
$$
\Lambda_{g_0,q_0}^{M_0}(\lambda-k^2): H^{3/2}(\partial M_0) \to H^{1/2}(\partial M_0)
$$
for all $k \in \mR$ from the knowledge of $(\partial T, \Lambda_{g_0,q_0}^T(\lambda))$ by Proposition \ref{lemma_dnmap_transversal_cylinder}. Since $\mu \mapsto \Lambda_{g_0,q_0}^{M_0}(\mu)$ is a meromorphic operator valued function whose poles are contained in $\{\lambda_1,\lambda_2,\ldots\}$ \cite[Lemma 4.5]{KKL}, this information determines $\Lambda_{g_0,q_0}^{M_0}(\mu)$ for all $\mu$ in the complex plane by analytic continuation. This is equivalent to knowing the DN map for the wave equation $\partial_t^2 - \Delta_{g_0} + q_0$ in $M_0 \times \{t > 0\}$ \cite[Chapter 4]{KKL}. The boundary control method then allows to construct a manifold isometric to $(M_0,g_0)$ and the potential $q_0$ from the DN map for the wave equation. See \cite{KKL}, \cite{KKLM} for more details.
\end{proof}

\section{Calder{\'o}n problem in a cylinder: continuous spectrum} \label{sec_bcmethod_second}

Assume the conditions in the first paragraph of Section \ref{sec_bcmethod_first}. We will next consider the case when $\lambda$ is in the continuous spectrum $[\lambda_1,\infty)$ but outside the set of thresholds $\{ \lambda_1, \lambda_2, \ldots \}$. In this case the Schr\"odinger equation in $T$ admits generalized eigenfunctions, and a radiation condition is required for uniqueness of solutions and for the definition of the DN map.

\begin{prop} \label{prop_wp_dirichlet_cylinder2}
Let $\lambda \in [\lambda_1, \infty) \setminus \{ \lambda_1,\lambda_2,\ldots \}$, choose $l_0 \geq 1$ so that $\lambda_{l_0} < \lambda < \lambda_{l_0+1}$, let $\delta > 1/2$, and let $m \geq 2$. Then for any $H^{m-1/2}_{\delta}(\partial T)$, the equation 
$$
(-\Delta+q_0-\lambda)u = 0 \ \ \text{in } T, \qquad u|_{\partial T} = f
$$
has a unique solution $u \in H^m_{-\delta}(T)$ satisfying the outgoing radiation condition
$$
(\partial_t \mp i\sqrt{\lambda-\lambda_l}) \widetilde{u}(t,l) \to 0 \quad \text{as $t \to \pm \infty$ for all $1 \leq l \leq l_0$}.
$$
If $f \in C^{\infty}_c(\partial T)$, then $u \in C^{\infty}(T)$ and there is a linear map 
$$
\Lambda_{g_0,q_0}^T(\lambda): C^{\infty}_c(\partial T) \to C^{\infty}(\partial T), \ \ f \mapsto \partial_{\nu} u|_{\partial T}.
$$
For any $\delta > 1/2$, this map extends as a bounded linear map 
$$
\Lambda_{g_0,q_0}^T(\lambda): H^{m-1/2}_{\delta}(\partial T) \to H^{m-3/2}_{-\delta}(\partial T).
$$
\end{prop}

Recall that when $\lambda \in \mC \setminus [\lambda_1,\infty)$, the main point in the reduction from the Calder\'on problem in the cylinder to the boundary control method was Proposition \ref{lemma_dnmap_transversal_cylinder}. This result states that 
$$
e^{ikt }\Lambda_{g_0,q_0}^{M_0}(\lambda-k^2) h = \Lambda_{g_0,q_0}^T(\lambda) (e^{ikt} h|_{\partial T}), \quad h \in H^{3/2}(\partial M_0).
$$
This identity does not directly generalize to the case where $\lambda$ is in the continuous spectrum, because the boundary value $e^{ikt} h|_{\partial T}$ on the right hand side is not in $H^{3/2}_{\delta}(\partial T)$ for $\delta > 1/2$. However, by using suitable cutoff and averaging arguments we can still recover the transversal DN maps from $\Lambda_{g_0,q_0}^T(\lambda)$.

\begin{prop}  \label{lemma_dnmap_transversal_cylinder2}
Let $\lambda \in [\lambda_1, \infty) \setminus \{ \lambda_1,\lambda_2,\ldots \}$, let $k \in \mR$, and assume that $\lambda-k^2 \notin \{ \lambda_1,\lambda_2,\ldots \}$. There is a family $(\Psi_R)_{R \geq 1} \subset C^{\infty}_c(\mR)$ with $\Psi_R(t) = 1$ for $\abs{t} \leq R$, such that 
$$
e^{ikt} \Lambda_{g_0,q_0}^{M_0}(\lambda-k^2) h = \lim_{R \to \infty} \frac{1}{R-1} \int_1^R \Lambda_{g_0,q_0}^T(\lambda) (e^{ikt} \Psi_{R'}(t) h|_{\partial T}) \,dR'
$$
pointwise on $\partial T$ for any $h \in C^{\infty}(\partial M_0)$.
\end{prop}

We now obtain an extension of Theorem \ref{thm_calderon_cylinder_reconstruction} to the case where $\lambda$ is in the continuous spectrum (but not in the set of thresholds). This result was stated as Theorem \ref{thm_calderon_cylinder_reconstruction2} in the introduction.

\begin{proof}[Proof of Theorem \ref{thm_calderon_cylinder_reconstruction2}]
The proof is exactly the same as that of Theorem \ref{thm_calderon_cylinder_reconstruction}, except that the use of Proposition \ref{lemma_dnmap_transversal_cylinder} is replaced by Proposition \ref{lemma_dnmap_transversal_cylinder2}.
\end{proof}

We now move to the proofs of Propositions \ref{prop_wp_dirichlet_cylinder2} and \ref{lemma_dnmap_transversal_cylinder2}. The first step is an existence and uniqueness result for the inhomogeneous Schr\"odinger equation in the cylinder.

\begin{lemma} \label{lemma_wpcylinder2}
Let $\lambda \in [\lambda_1, \infty) \setminus \{ \lambda_1,\lambda_2,\ldots \}$ and choose $l_0 \geq 1$ so that $\lambda_{l_0} < \lambda < \lambda_{l_0+1}$. Let $\delta > 1/2$, let $\mu \in \mR$, and let $m \geq 0$ be an integer. For any $F = F_1 + F_2$ where $F_1 \in H^m_{\delta}(T)$, $F_2 \in H^m_{\mu}(T)$ and 
\begin{gather*}
F_1(t,x) = \sum_{l=1}^{l_0} \widetilde{F}_1(t,l) \phi_l(x), \qquad F_2(t,x) = \sum_{l=l_0+1}^{\infty} \widetilde{F}_2(t,l) \phi_l(x),
\end{gather*}
there is a solution $u = u_1 + u_2$ of the equation 
$$
(-\partial_t^2 - \Delta_{g_0} + q_0 - \lambda) u = F \quad \text{in } T
$$
where $u_1 \in H^{m+2}_{-\delta}(T) \cap H^1_{-\delta,0}(T)$ and $u_2 \in H^{m+2}_{\mu}(T) \cap H^1_{\mu,0}(T)$ are of the form $u_1 = \sum_{l=1}^{l_0} \widetilde{u}_1(\,\cdot\,,l) \phi_l $ and $u_2 = \sum_{l=l_0+1}^{\infty} \widetilde{u}_2(\,\cdot\,,l) \phi_l$. Further, 
$$
\norm{u_1}_{H^{m+2}_{-\delta}(T)} \leq C \norm{F_1}_{H^m_{\delta}(T)}, \quad \norm{u_2}_{H^{m+2}_{\mu}(T)} \leq C \norm{F_2}_{H^m_{\mu}(T)}.
$$
The solution is unique up to an element of the form 
$$
\sum_{l=1}^{l_0} c_l^{\pm} e^{\pm i \sqrt{\lambda-\lambda_l} t} \phi_l(x)
$$
where $c_l^{\pm}$ are constants. If one assumes the outgoing radiation condition 
$$
(\partial_t \mp i\sqrt{\lambda-\lambda_l}) \widetilde{u}(t,l) \to 0 \quad \text{as $t \to \pm \infty$ for all $1 \leq l \leq l_0$},
$$
then the solution $u$ is unique.
\end{lemma}
\begin{proof}
For uniqueness, suppose that $u \in H^1_{r,0}(T)$ for some real number $r$ solves $(-\partial_t^2-\Delta_{g_0}+q_0-\lambda)u = 0$ in $T$. Then the partial Fourier coefficients $\widetilde{u}(t,l)$ satisfy 
$$
(-\partial_t^2 + \lambda_l - \lambda) \widetilde{u}(\,\cdot\,,l) = 0  \quad \text{in $\mR$, for all $l \geq 1$}.
$$
If $l \geq l_0 + 1$, then $\lambda_l-\lambda > 0$ and by taking Fourier transforms we see that the only tempered distribution solving this equation on $\mR$ is zero. Thus $\widetilde{u}(\,\cdot\,,l) = 0$ for $l \geq l_0+1$. If $1 \leq l \leq l_0$, then $\lambda_l - \lambda < 0$ and an easy argument shows that the only distributional solution of the above equation for $\widetilde{u}(\,\cdot\,,l)$ is 
$$
\widetilde{u}(t,l) = c_l^+ e^{i\sqrt{\lambda-\lambda_l} t} + c_l^- e^{-i\sqrt{\lambda-\lambda_l} t}
$$
for some constants $c_l^{\pm}$. If the radiation condition holds it follows that $c_l^{\pm} = 0$, concluding the proof of uniqueness.

For existence, let first $m=0$. We define 
$$
\widetilde{u}_2(\eta,l) = \mF_{\eta}^{-1} \left\{ \frac{1}{\eta^2+\lambda_l-\lambda} \widehat{F}_2(\eta,l) \right\}, \quad l \geq l_0+1.
$$
Since $\lambda_l-\lambda > 0$ for $l \geq l_0+1$, the proof of Lemma \ref{lemma_wpcylinder1} shows that the function $u_2 = \sum_{l=l_0+1}^{\infty} \widetilde{u}_2(\,\cdot\,,l) \phi_l$ solves the equation $(-\partial_t^2-\Delta_{g_0}+q_0-\lambda)u_2 = F_2$ and has the required properties. The function $u_1$ is obtained as $u_1 = \sum_{l=1}^{l_0} \widetilde{u}_1(\,\cdot\,,l) \phi_l$, where $\widetilde{u}_1(\,\cdot\,,l)$ should satisfy 
$$
(-\partial_t^2 + \lambda_l - \lambda) \widetilde{u}_1(\,\cdot\,,l) = \widetilde{F}_1(\,\cdot\,,l)  \quad \text{in $\mR$, for $1 \leq l \leq l_0$}.
$$
We choose the solution 
$$
\widetilde{u}_1(\,\cdot\,,l) = R_0(\lambda-\lambda_l+i0) \widetilde{F}_1(\,\cdot\,,l)
$$
where $R_0(z) = (-\partial_t^2 - z)^{-1}$ is the resolvent of the Laplacian on the real line, and $R_0(s + i0)$ is the outgoing resolvent at energy $s > 0$. If $k > 0$ one has the well known formula (which follows from a direct computation) 
$$
(R_0(k^2+i0)f)(t) = \int_{-\infty}^{\infty} G(t-t') f(t') \,dt', \quad G(t) = \frac{i}{2k} e^{ik\abs{x}}.
$$
Agmon's limiting absorption principle (see \cite[Section 14.3]{H3}, \cite[Section XIII.8]{ReedSimon}) gives that $\widetilde{u}_1(\,\cdot\,,l) \in H^2_{-\delta}(\mR)$, and then $u_1$ is a solution of $(-\partial_t^2-\Delta_{g_0}+q_0-\lambda)u_1 = F_1$ with the required properties (also satisfying the outgoing radiation condition). This concludes the proof for the case $m=0$.

We show the case of general $m$ by induction: the statement has been proved for $m=0$, and we assume that it holds for values up to $m-1$. Let $F_2 \in H^m_{\mu}(T)$ (the proof for $F_1$ is analogous). Then also $F_2 \in H^{m-1}_{\mu}(T)$, and by the inductive hypothesis there is a unique solution $u \in H^{m+1}_{\mu}(T)$ of 
$$
(-\partial_t^2-\Delta_{g_0}+q_0-\lambda)u = F_2 \text{ in } T, \quad u|_{\partial T} = 0.
$$
Then $\partial_t^j u$ satisfies the same equation with right hand side $\partial_t^j F_2 \in H^{m-j}_{\mu}(T)$, so we have $\partial_t^j u \in H^{m+2-j}_{\mu}(T)$ for $j \geq 1$ with 
$$
\norm{\partial_t^j u}_{H^{m+2-j}_{\mu}(T)} \leq C \norm{F_2}_{H^{m}_{\mu}(T)}.
$$
The equation implies that for almost every $t$, 
$$
-\Delta_{g_0} u(t,\,\cdot\,) = (\partial_t^2-q_0+\lambda) u(t,\,\cdot\,) + F_2(t,\,\cdot\,) \text{ in } M_0.
$$
Since also $u(t,\,\cdot\,)|_{\partial M_0} = 0$, elliptic regularity implies that 
$$
\norm{u(t,\,\cdot\,)}_{H^{m+2}(M_0)} \leq C \norm{(\partial_t^2-q_0+\lambda) u(t,\,\cdot\,) + F_2(t,\,\cdot\,)}_{H^m(M_0)}.
$$
Multiplying by $\br{t}^{\mu}$ and taking the $L^2$ norm over $\mR$, we obtain 
$$
\norm{u}_{L^2_{\mu}(\mR;H^{m+2}(M_0))} \leq C (\norm{u}_{H^m_{\mu}(T)}+\norm{F_2}_{H^m_{\mu}(T)}) \leq C \norm{F_2}_{H^m_{\mu}(T)}.
$$
Thus $u, \partial_t u, \nabla_x u \in H^{m+1}_{\mu}(T)$, showing that $u \in H^{m+2}_{\mu}(T)$ with the right bounds.
\end{proof}

\begin{proof}[Proof of Proposition \ref{prop_wp_dirichlet_cylinder2}]
Let $\delta > 1/2$, and suppose that $f \in H^{m-1/2}_{\delta}(\partial T)$. Choose $E^f \in H^m_{\delta}(T)$ with $E^f|_{\partial T} = f$ and $\norm{E^f}_{H^m_{\delta}(T)} \leq C \norm{f}_{H^{m-1/2}_{\delta}(\partial T)}$. We look for a solution of 
$$
(-\Delta+q_0-\lambda)u = 0 \ \ \text{in } T, \qquad u|_{\partial T} = f
$$
having the form $u = E^f + w$. Thus, we obtain the equivalent equation 
$$
(-\Delta+q_0-\lambda)w = F \ \ \text{in } T, \qquad w|_{\partial T} = 0
$$
where $F = -(-\Delta+q_0-\lambda) E^f$. Since $\norm{F}_{H^{m-2}_{\delta}(T)} \leq C \norm{f}_{H^{m-1/2}_{\delta}(\partial T)}$, Lemma \ref{lemma_wpcylinder2} shows that there is a unique solution $w \in H^m_{-\delta}(T) \cap H^1_{-\delta,0}(T)$ satisfying  
$$
(\partial_t \mp i\sqrt{\lambda-\lambda_l}) \widetilde{w}(t,l) \to 0 \quad \text{as $t \to \pm \infty$ for all $1 \leq l \leq l_0$}.
$$
One also has $\norm{w}_{H^m_{-\delta}(T)} \leq C \norm{f}_{H^{m-1/2}_{\delta}(\partial T)}$. Thus we have a unique solution $u = E^f + w \in H^m_{-\delta}(T)$ to the original problem, satisfying the same radiation condition as $w$ since $(\partial_t \mp i\sqrt{\lambda-\lambda_l}) \widetilde{E}^f(\,\cdot\,,l) \in H^1_{\delta}(\mR)$ for all $l$. We also have 
$$
\norm{u}_{H^m_{-\delta}(T)} \leq C \norm{f}_{H^{m-1/2}_{\delta}(\partial T)}.
$$
The result follows.
\end{proof}

Before the proof of Proposition \ref{lemma_dnmap_transversal_cylinder2}, we record some further properties of solutions of the Schr\"odinger equation in the cylinder having boundary values of the form 
$$
f(t,x) = a(t)h(x)
$$
where $h \in H^{m-1/2}(\partial M_0)$. Given $v\in L^2_\mu (\R\times M_0)$, we define 
\begin{gather*}
P_1 v(t,x) = \sum_{l=1}^{l_0} \widetilde{v}(t,l) \phi_l(x), \qquad P_2 v(t,x) = \sum_{l=l_0+1}^{\infty} \widetilde{v}(t,l) \phi_l(x),
\end{gather*}
where $\tilde v(t,l)= (v(t,\cdotp),\phi_l)_{L^2(M_0)}$.

\begin{lemma} \label{lemma_solutions_cylinder_properties}
Assume that $\lambda \in [\lambda_1, \infty) \setminus \{ \lambda_1,\lambda_2,\ldots \}$, choose $l_0 \geq 1$ so that $\lambda_{l_0} < \lambda < \lambda_{l_0+1}$, let $m \geq 2$, and let $E_0: H^{m-1/2}(\partial M_0) \to H^m(M_0)$ be a bounded extension operator. Let $f(t,x) = a(t) h(x)$ where $a \in H^m_{\mu}(\mR)$ with $\mu \in \mR$ and $h \in H^{m-1/2}(\partial M_0)$, and define 
$$
E^f(t,x) = a(t) E_0 h(x), \qquad F^f=- (-\partial_t^2 - \Delta_{g_0} + q_0 - \lambda) E^f.
$$
Also define 
$$
E_j(\cdotp;f) = P_j E^f, \qquad F_j(\cdotp;f) = P_j F^f.
$$

If $\mu > 1/2$, denote by $u(t,x;f)$ the solution of the Dirichlet problem 
\begin{gather*} 
(-\partial_t^2 - \Delta_{g_0} + q_0 - \lambda) u = 0 \text{ in } T, \qquad u|_{\p T}=f,
\\ \nonumber
(\partial_t \mp i\sqrt{\lambda-\lambda_l}) \widetilde{u}(t,l;f) \to 0 \quad \text{as $t \to \pm \infty$ for all $1 \leq l \leq l_0$},
\end{gather*}
where $\tilde u(t,l;f)=(u(t,\cdotp;f),\phi_l)_{L^2(M_0)}$. Then we have 
$$
u = u_1 + u_2, \qquad u_j = E_j + w_j,
$$
where $w_j = w_j(\cdotp;f)$ are the solutions of 
\begin{gather*} 
(-\partial_t^2 - \Delta_{g_0} + q_0 - \lambda) w_1= F_1(\cdotp;f) \text{ in } T, \qquad w_1|_{\p T}=0, \\
(\partial_t \mp i\sqrt{\lambda-\lambda_l}) \widetilde{w}_1(t,l) \to 0 \quad \text{as $t \to \pm \infty$ for all $1 \leq l \leq l_0$},
\end{gather*}
with $\tilde w_j(t,l)= (w_j(t,\cdotp),\phi_l)_{L^2(M_0)}$, and
\begin{gather*} 
(-\partial_t^2 - \Delta_{g_0} + q_0 - \lambda) w_2= F_2(\cdotp;f) \quad \text{in } T, \qquad w_2|_{\p T}=0.
\end{gather*}
If $a \in H^m_{\delta}(\mR)$ with $\delta > 1/2$, then the equation for $w_1$ has a unique solution $w_1 \in H^m_{-\delta}(T) \cap H^1_{-\delta,0}(T)$ with $w_1 \in \hbox{Ran}(P_1)$. Similarly, if $a \in H^m_{\mu}(\mR)$ for some $\mu \in \mR$, then the equation for $w_2$ has a unique solution $w_2 \in H^m_{\mu}(T) \cap H^1_{\mu,0}(T)$ with $w_2 \in \hbox{Ran}(P_2)$. We have the norm estimates 
$$
\norm{E_j}_{H^m_{\mu}(T)} + \norm{F_j}_{H^{m-2}_{\mu}(T)} \leq C \norm{a}_{H^m_{\mu}(\mR)} \norm{h}_{H^{m-1/2}(\partial M_0)}, \quad \mu \in \mR,
$$
and 
\begin{gather*}
\norm{u_1}_{H^m_{-\delta}(T)} + \norm{w_1}_{H^m_{-\delta}(T)} \leq C \norm{a}_{H^m_{\mu}(\mR)} \norm{h}_{H^{m-1/2}(\partial M_0)}, \quad \delta > 1/2, \\
\norm{u_2}_{H^m_{\mu}(T)} + \norm{w_2}_{H^m_{\mu}(T)} \leq C \norm{a}_{H^m_{\mu}(\mR)} \norm{h}_{H^{m-1/2}(\partial M_0)}, \quad \mu \in \mR.
\end{gather*}
\end{lemma}
\begin{proof}
We note the estimate 
$$
\norm{a(t) \varphi(x)}_{H^m_{\mu}(T)} \leq C \norm{a}_{H^m_{\mu}(\mR)} \norm{\varphi}_{H^m(M_0)}.
$$
Consequently 
$$
\norm{E^f}_{H^m_{\mu}(T)} + \norm{F^f}_{H^{m-2}_{\mu}(T)} \leq C \norm{a}_{H^m_{\mu}(\mR)} \norm{h}_{H^{m-1/2}(\partial M_0)}.
$$
The same estimates are true for $E_j$ and $F_j$, since the projections $P_j$ commute with $\partial_t$, $\Delta_{g_0}$ and with multiplication by $\br{t}^{\mu}$. The result now follows from Lemma \ref{lemma_wpcylinder2} and the standard reduction from the Dirichlet problem to an inhomogeneous problem with zero boundary values.
\end{proof}

\begin{proof}[Proof of Proposition \ref{lemma_dnmap_transversal_cylinder2}]
Assume that $\lambda \in [\lambda_1, \infty) \setminus \{ \lambda_1,\lambda_2,\ldots \}$, and choose $l_0 \geq 1$ so that $\lambda_{l_0} < \lambda < \lambda_{l_0+1}$. Fix $k \in \mR$ so that $\lambda-k^2$ is not a Dirichlet eigenvalue of $-\Delta_{g_0}+q_0$ in $M_0$. Also fix an integer $m > \dim(T)/2 + 1$.

We will show that for any $h \in H^{m-1/2}(\partial M_0)$, one has the pointwise limit 
$$
e^{ikt} \Lambda_{g_0,q_0}^{M_0}(\lambda-k^2) h = \lim_{R \to \infty} \frac{1}{R-1} \int_1^R \Lambda_{g_0,q_0}^T(\lambda)f_{R'} \,dR'
$$
where $f_R$ is the function on $\partial T$ given by 
$$
f_R(t,x) = e^{ikt} \Psi_{R}(t) h(x)
$$
and $\Psi_R(t) \in C_c^{\infty}(\mR)$ are suitable cutoffs. Below, we will use the notations in Lemma \ref{lemma_solutions_cylinder_properties}. For later purposes we choose the extension operator $E_0: H^{m-1/2}(\partial M_0) \to H^m(M_0)$ in Lemma \ref{lemma_solutions_cylinder_properties} to be $E_0: h \mapsto v(x;h)$, where $v(x;h)$ is the unique solution of the problem 
$$
(-\Delta_{g_0}+q_0-(\lambda-k^2)) v(x;h) = 0 \ \ \text{in } M_0, \quad v(\cdotp;h)|_{\partial M_0} = h.
$$
The proof below will make use of the splitting 
$$
u = u_1 + u_2,
$$
and also the splitting 
$$
v = v_1 + v_2
$$
where $v_j(\cdotp;h) \in H^m(M_0)$ are the projections $v_j(\cdotp;h) = Q_j v(\cdotp;h)$. Here, for $V \in L^2(M_0)$, 
\begin{gather*}
Q_1 V = \sum_{l=1}^{l_0} \widetilde{V}(l) \phi_l(x), \quad Q_2 V = \sum_{l=l_0+1}^{\infty} \widetilde{V}(l) \phi_l(x),
\end{gather*}
with $\tilde V(l)=( V, \phi_l )_{L^2(M_0)}$. In fact, we will prove that 
\begin{gather}
\lim_{R \to \infty} \frac{1}{R-1} \int_1^R \partial_{\nu} u_1(\cdotp;f_{R'}) \,dR' = e^{ikt} \partial_{\nu} v_1(\cdotp;h)|_{\partial T}, \label{averaged_eq1} \\
\lim_{R \to \infty} \frac{1}{R-1} \int_1^R \partial_{\nu} u_2(\cdotp;f_{R'}) \,dR' = e^{ikt} \partial_{\nu} v_2(\cdotp;h)|_{\partial T}. \label{averaged_eq2}
\end{gather}
Note that 
\begin{gather*}
\Lambda_{g_0,q_0}^T(\lambda) f_R = \partial_{\nu} u(t,x;f_R) \big|_{\partial T}, \\
\Lambda_{g_0,q_0}^{M_0}(\lambda-k^2) h = \partial_{\nu} v(\cdotp;h)|_{\partial M_0}.
\end{gather*}
Thus the proposition will follow immediately from \eqref{averaged_eq1} and \eqref{averaged_eq2}.

Let us next describe the cutoff functions. If $R \to \infty$, the boundary value $e^{ikt} \Psi_R(t) h(x)$ converges to $e^{ikt} h(x)$, a function in $H^{m-1/2}_{\mu}(\partial T)$ for $\mu < -1/2$. Fix some $\mu < -1/2$, and let $\psi_R(t)=1$ for $|t|<R$, $\psi_R(t)=0$ for $|t|>R$.  We approximate the functions $\psi_R(t)$ by $\Psi_R(t)\in C^\infty_c((-R-1,R+1))$ that are functions 
for which
\begin{equation} \label{eq: norms}
\lim_{R\to \infty }\|\Psi_R-\psi_R\|_{L^1(\R)}=0, \quad
\lim_{R\to \infty }\|\Psi_R-1\|_{H^m_{\mu}(\R)}= 0.
\end{equation}
Such functions can be chosen to be 
\ba
\Psi_R(t)=\left\{\begin{array}{ll}1&\hbox{for $|t|\leq R$},\\
\Phi (R^{\alpha}(|t|-R))&\hbox{for $|t|>R$},
\end{array}\right.
\ea
where  $\Phi \in C^\infty_c((-1,1))$ is equal to one near $0$, and $\alpha$ is a positive constant chosen so that $m\alpha+\mu+1/2< 0$. The norm bounds as $R \to \infty$ are valid because $\Psi_R-\psi_R$ is supported in $R \leq \abs{t} \leq R+R^{-\alpha}$, $\Psi_R-1$ is supported in $\abs{t} \geq R$, $\norm{\Psi_R}_{W^{m,\infty}} \leq C R^{m\alpha}$, and $(\int_R^{\infty} t^{2\mu} \,dt)^{1/2} \leq C R^{\mu+1/2}$.

Let us denote
\ba
f_R(t,x)=e^{ikt} \Psi_R(t) h(x),\quad f(t,x)=e^{ikt} h(x)
\ea
where $h \in H^{m-1/2}(\partial M_0)$. We will now prove \eqref{averaged_eq2}. Note that by construction we have 
$$
E^{f_R}-E^f = E^{f_R-f}
$$
and thus $E_j(\,\cdot\,,f_R)-E_j(\,\cdot\,,f) = E_j(\,\cdot\,,f_R-f)$. The function $F_j$, $w_j$, and $u_j$ satisfy a similar property. Using Lemma \ref{lemma_solutions_cylinder_properties}, we see that 
\begin{align*} 
 &\norm{w_2(\cdotp;f_R)-w_2(\cdotp;f)}_{H^m_{\mu}(T)} = \norm{w_2(\cdotp;f_R-f)}_{H^m_{\mu}(T)} \\
 &\quad \leq C  \norm{e^{ikt}(\Psi_R-1)}_{H^m_{\mu}(\mR)} \norm{h}_{H^{m-1/2}(\partial M_0)}.
\end{align*}
Similarly,
\begin{align*}
 &\norm{E_2(\cdotp;f_R)-E_2(\cdotp;f)}_{H^m_{\mu}(T)} = \norm{E_2(\cdotp;f_R-f)}_{H^m_{\mu}(T)} \\
 &\quad \leq C \norm{e^{ikt}(\Psi_R-1)}_{H^m_{\mu}(\mR)} \norm{h}_{H^{m-1/2}(\partial M_0)}.
\end{align*}
Since $u_2 = E_2 + w_2$, the estimate (\ref{eq: norms}) implies that 
$$
\norm{u_2(\cdotp;f_R)-u_2(\cdotp;f)}_{H^m_{\mu}(T)} \to 0
$$
and consequently 
$$
\lim_{R \to \infty} \partial_{\nu} u_2(\cdotp;f_R) = \partial_{\nu} u_2(\cdotp;f)
$$
in $H^{m-3/2}_{\mu}(\partial T)$. By Sobolev embedding this limit also holds pointwise, and we also have 
\begin{equation} \label{averaged_eq3}
\lim_{R \to \infty} \frac{1}{R-1} \int_1^R \partial_{\nu} u_2(\cdotp;f_{R'}) \,dR' = \partial_{\nu} u_2(\cdotp;f).
\end{equation}

We can connect the last expression to the $v_2$ component of the time-harmonic solution $v = v(\cdotp;h)$. As in Proposition \ref{lemma_dnmap_transversal_cylinder}, we see that the function $e^{ikt} v_2(x;h)\in H^2_\mu(T)$ solves the equation $(-\partial_t^2-\Delta_{g_0}+q_0-\lambda) u = 0$ in $T$ with boundary value $e^{ikt} v_2|_{\partial T}$. Similarly, the function $u_2 = u_2(\cdotp;f)$ solves the same equation with boundary value $e^{ikt} (Q_2 E_0 h)|_{\partial T}$, where $E_0$ was the bounded extension operator. But since we chose $E_0 h(x) = v(x;h)$, the two solutions have the same boundary values and by uniqueness one has 
\ba
& &u_2(t,x;f)=e^{ikt} v_2(x;h).
\ea
Together with \eqref{averaged_eq3}, this proves \eqref{averaged_eq2}.

It remains to show the identity \eqref{averaged_eq1} concerning $u_1$ and $v_1$. Recall that 
$$
u_1(t,x;f_R) = \sum_{l=1}^{l_0} \widetilde{u}(t,l;f_R) \phi_l(x).
$$
We start by giving formulae for Fourier coefficients $\widetilde{u}(t,l;f_R)$ for $1 \leq l \leq l_0$. If $\eta \in H^{3/2}(\partial T)$ is compactly supported in the $t$ variable, we have
\begin{align*}
0&=( (-\partial_t^2 - \Delta_{g_0} + q_0 - \lambda)u(t,\cdotp;\eta),\phi_l )_{L^2(M_0)} \\
 &=(-\p_t^2+\lambda_l-\lambda)\tilde u(t,l;\eta)+\int_{\p M_0}\eta(t,y) \partial_{\nu} \phi_l(y)\,dS(y).
\end{align*}
Since $\lambda-\lambda_l > 0$ for $1 \leq l \leq l_0$, this and the radiation condition imply that 
\ba
\tilde u(t,l;\eta)=\int_{-\infty}^{\infty} G_l(t-t^\prime)\tilde \eta(t^\prime,l)\,dt^\prime,
\ea
where 
\ba
& &G_l(t)=\frac 1{2i} (\lambda-\lambda_l)^{-1/2}e^{i|t|\sqrt{\lambda-\lambda_l}},
\\
& &\tilde \eta(t,l)= \int_{\p M_0} \eta(t,y) \partial_{\nu} \phi_l(y) \,dS(y).
\ea

Consider the function $I_1(R) = I_1(t,x;R)$ in $H^{1/2}_{-\delta}(\partial T)$ for any $\delta > 1/2$, given by 
$$
I_1(R) = \partial_{\nu} u_1(t,x;f_R) = \sum_{l=1}^{l_0} \widetilde{u}(t,l;f_R) \partial_{\nu} \phi_l(x).
$$
Using the expression for Fourier coefficients above, we have 
\begin{align*}
&I_1(R) = \sum_{l=1}^{l_0} \frac{(\lambda-\lambda_l)^{-1/2}}{2i} \left[ \int_{-\infty}^{\infty} e^{i\abs{t-t'} \sqrt{\lambda-\lambda_l}} \widetilde{f}_R(t',l) \,dt' \right] \partial_{\nu} \phi_l(x) \\ 
&= \sum_{l=1}^{l_0} \frac{(\lambda-\lambda_l)^{-1/2}}{2i} \left[ \int_{\partial T} e^{i\abs{t-t'} \sqrt{\lambda-\lambda_l}} f_R(t',y) \partial_{\nu} \phi_l(y) \,dt' \,dS(y) \right] \partial_{\nu} \phi_l(x) \\
&= \sum_{l=1}^{l_0} \frac{(\lambda-\lambda_l)^{-1/2}}{2i} \left[ \int_{\partial T} e^{i\abs{t-t'} \sqrt{\lambda-\lambda_l}} e^{ikt'} \Psi_R(t') h(y) \partial_{\nu} \phi_l(y) \,dt' \,dS(y) \right] \partial_{\nu} \phi_l(x).
\end{align*}
Replacing here $\Psi_R$ by  $\psi_R$ and using the first estimate in \eqref{eq: norms} results in an $o(1)$ error in $L^{\infty}(\partial T)$ as $R \to \infty$. We thus obtain 
\begin{align*}
I_1(R) =  \sum_{l=1}^{l_0} \frac{(\lambda-\lambda_l)^{-1/2}}{2i} \left[ \int_{-R}^{R} e^{i\abs{t-t'} \sqrt{\lambda-\lambda_l}} e^{ikt'} \,dt' \right] \widetilde{h}(l) \partial_{\nu} \phi_l(x) + o(1)
\end{align*}
where 
$$
\widetilde{h}(l) = \int_{M_0} h(y) \partial_{\nu} \phi_l(y) \,dS(y).
$$
For a given $t$, we assume $R$ so large that $t \in (-R,R)$. The $t'$ integral can be computed explicitly, and we obtain 
\begin{multline*}
\int_{-R}^{R} e^{i\abs{t-t'} \sqrt{\lambda-\lambda_l}} e^{ikt'} \,dt' = \frac{2i (\lambda-\lambda_l)^{1/2}}{\lambda-\lambda_l-k^2} e^{ikt} \\
 + \frac{e^{i(k+\sqrt{\lambda-\lambda_l})R-it\sqrt{\lambda-\lambda_l}}}{i(k+\sqrt{\lambda-\lambda_l})} - \frac{e^{-i(k-\sqrt{\lambda-\lambda_l})R+it\sqrt{\lambda-\lambda_l}}}{i(k-\sqrt{\lambda-\lambda_l})}.
 \end{multline*}
The last two terms oscillate with respect to $R$, but we can remove these oscillating terms by averaging: since by assumption $k \pm \sqrt{\lambda-\lambda_l} \neq 0$, we have 
$$
\lim_{R \to \infty} \frac{1}{R-1} \int_1^R e^{i(k \pm \sqrt{\lambda-\lambda_l})R'} \,dR' = 0.
$$
This shows that for any fixed $(t,x) \in \partial T$, we have 
$$
\lim_{R \to \infty} \frac{1}{R-1} \int_1^R I_1(R') \,dR' = \sum_{l=1}^{l_0} \frac{1}{{\lambda-\lambda_l-k^2}} e^{ikt} \widetilde{h}(l) \partial_{\nu} \phi_l(x).
$$
We can relate the last expression to the time-harmonic solutions $v(\cdotp;h)$ of $(-\Delta_{g_0}+q_0-(\lambda-k^2)) v = 0$ in $M_0$ with $v|_{\partial M_0} = h$. We have
\begin{align*}
0&=( (- \Delta_{g_0} + q_0 - (\lambda-k^2))v(\cdotp;h),\phi_l)_{L^2(M_0)}\\
&=(\lambda_l - (\lambda-k^2))\tilde v(l;h) + \int_{\p M_0} h(y) \partial_{\nu} \phi_l(y) \,dS(y)
\end{align*}
which implies
\ba
 \widetilde{h}(l) = (\lambda-\lambda_l-k^2) \widetilde{v}(l;h).
\ea
This shows that 
\begin{align*}
\lim_{R \to \infty} \frac{1}{R-1} \int_1^R \partial_{\nu} u_1(t,x;f_{R'}) \,dR' &= \sum_{l=1}^{l_0} e^{ikt} \widetilde{v}(l;h) \partial_{\nu} \phi_l(x) \\
 &= e^{ikt} \partial_{\nu} v_1(\,\cdot\,;h).
\end{align*}
This shows \eqref{averaged_eq1}, which concludes the proof.
\end{proof}


\providecommand{\bysame}{\leavevmode\hbox to3em{\hrulefill}\thinspace}
\providecommand{\href}[2]{#2}

\end{document}